\documentclass[a4paper,11pt]{amsart}
\usepackage{mathrsfs}
\usepackage{cases}
\usepackage{epic}
\usepackage{amsfonts}
\usepackage{graphicx}
\usepackage{amsmath}
\usepackage{amssymb}
\usepackage{pdflscape}
\usepackage[all]{xypic}
\usepackage[all]{xy}
\usepackage{color}
\usepackage{colordvi}
\usepackage{multicol}
\usepackage[linktocpage=true]{hyperref}
\hypersetup{colorlinks,linkcolor=blue,urlcolor=cyan,citecolor=blue}

\begin{document}
\input xy
\xyoption{all}

\newcommand{\iadd}{\operatorname{iadd}\nolimits}
\newcommand{\gid}{\operatorname{grinj.dim}\nolimits}

\renewcommand{\mod}{\operatorname{mod}\nolimits}
\newcommand{\proj}{\operatorname{proj}\nolimits}
\newcommand{\inj}{\operatorname{inj.}\nolimits}
\newcommand{\rad}{\operatorname{rad}\nolimits}
\newcommand{\soc}{\operatorname{soc}\nolimits}
\newcommand{\ind}{\operatorname{inj.dim}\nolimits}
\newcommand{\Ginj}{\operatorname{Ginj}\nolimits}
\newcommand{\Mod}{\operatorname{Mod}\nolimits}
\newcommand{\R}{\operatorname{R}\nolimits}
\newcommand{\End}{\operatorname{End}\nolimits}
\newcommand{\colim}{\operatorname{colim}\nolimits}
\newcommand{\gldim}{\operatorname{gl.dim}\nolimits}
\newcommand{\cone}{\operatorname{cone}\nolimits}
\newcommand{\rep}{\operatorname{rep}\nolimits}
\newcommand{\Ext}{\operatorname{Ext}\nolimits}
\newcommand{\Tor}{\operatorname{Tor}\nolimits}
\newcommand{\Hom}{\operatorname{Hom}\nolimits}
\newcommand{\Top}{\operatorname{top}\nolimits}
\newcommand{\supp}{\operatorname{supp}\nolimits}
\newcommand{\red}{\operatorname{red}\nolimits}
\newcommand{\rdim}{\operatorname{\bf r\text{-}dim}\nolimits}
\newcommand{\hdim}{\operatorname{\bf h\text{-}dim}\nolimits}

\newcommand{\Coker}{\operatorname{Coker}\nolimits}
\newcommand{\thick}{\operatorname{thick}\nolimits}
\newcommand{\rank}{\operatorname{rank}\nolimits}
\newcommand{\Gproj}{\operatorname{Gproj}\nolimits}
\newcommand{\Len}{\operatorname{Length}\nolimits}
\newcommand{\RHom}{\operatorname{RHom}\nolimits}
\renewcommand{\deg}{\operatorname{deg}\nolimits}
\renewcommand{\Im}{\operatorname{Im}\nolimits}
\newcommand{\Ker}{\operatorname{Ker}\nolimits}
\newcommand{\Coh}{\operatorname{Coh}\nolimits}
\newcommand{\Id}{\operatorname{Id}\nolimits}
\newcommand{\dimv}{\operatorname{\mathbf{dim}}\nolimits}
\newcommand{\Aus}{\operatorname{Aus}\nolimits}
\newcommand{\CM}{\operatorname{CM}\nolimits}
\newcommand{\For}{\operatorname{{\bf F}or}\nolimits}
\newcommand{\coker}{\operatorname{Coker}\nolimits}
\renewcommand{\dim}{\operatorname{dim}\nolimits}
\renewcommand{\div}{\operatorname{div}\nolimits}
\newcommand{\Ab}{{\operatorname{Ab}\nolimits}}
\newcommand{\diag}{{\operatorname{diag}\nolimits}}
\newcommand{\dg}{{\operatorname{dg}\nolimits}}

\renewcommand{\Vec}{{\operatorname{Vec}\nolimits}}
\newcommand{\pd}{\operatorname{proj.dim}\nolimits}
\newcommand{\gr}{\operatorname{gr}\nolimits}
\newcommand{\id}{\operatorname{inj.dim}\nolimits}
\newcommand{\Gd}{\operatorname{G.dim}\nolimits}
\newcommand{\Ind}{\operatorname{ind}\nolimits}
\newcommand{\add}{\operatorname{add}\nolimits}
\newcommand{\pr}{\operatorname{pr}\nolimits}
\newcommand{\oR}{\operatorname{R}\nolimits}
\newcommand{\oL}{\operatorname{L}\nolimits}
\newcommand{\Perf}{{\mathfrak Perf}}
\newcommand{\cc}{{\mathcal C}}
\newcommand{\gc}{{\mathcal GC}}
\newcommand{\ce}{{\mathcal E}}
\newcommand{\cs}{{\mathcal S}}
\newcommand{\cf}{{\mathcal F}}
\newcommand{\cx}{{\mathcal X}}
\newcommand{\ct}{{\mathcal T}}
\newcommand{\cu}{{\mathcal U}}
\newcommand{\cv}{{\mathcal V}}
\newcommand{\cn}{{\mathcal N}}
\newcommand{\mcr}{{\mathcal R}}
\newcommand{\ch}{{\mathcal H}}
\newcommand{\ca}{{\mathcal A}}
\newcommand{\cb}{{\mathcal B}}
\newcommand{\ci}{{\mathcal I}}
\newcommand{\cj}{{\mathcal J}}
\newcommand{\cm}{{\mathcal M}}
\newcommand{\cp}{{\mathcal P}}
\newcommand{\cg}{{\mathcal G}}
\newcommand{\cw}{{\mathcal W}}
\newcommand{\co}{{\mathcal O}}
\newcommand{\cq}{{\mathcal Q}}
\newcommand{\cd}{{\mathcal D}}
\newcommand{\ck}{{\mathcal K}}
\newcommand{\calr}{{\mathcal R}}
\newcommand{\ol}{\overline}
\newcommand{\ul}{\underline}
\newcommand{\st}{[1]}
\newcommand{\ow}{\widetilde}
\renewcommand{\P}{\mathbf{P}}
\newcommand{\pic}{\operatorname{Pic}\nolimits}
\newcommand{\Spec}{\operatorname{Spec}\nolimits}
\newtheorem{theorem}{Theorem}[section]
\newtheorem{acknowledgement}[theorem]{Acknowledgement}
\newtheorem{algorithm}[theorem]{Algorithm}
\newtheorem{axiom}[theorem]{Axiom}
\newtheorem{case}[theorem]{Case}
\newtheorem{claim}[theorem]{Claim}
\newtheorem{conclusion}[theorem]{Conclusion}
\newtheorem{assumption}[theorem]{Assumption}
\newtheorem{conjecture}[theorem]{Conjecture}
\newtheorem{construction}[theorem]{Construction}
\newtheorem{corollary}[theorem]{Corollary}
\newtheorem{criterion}[theorem]{Criterion}
\newtheorem{definition}[theorem]{Definition}
\newtheorem{example}[theorem]{Example}
\newtheorem{exercise}[theorem]{Exercise}
\newtheorem{lemma}[theorem]{Lemma}
\newtheorem{notation}[theorem]{Notation}
\newtheorem{problem}[theorem]{Problem}
\newtheorem{proposition}[theorem]{Proposition}
\newtheorem{remark}[theorem]{Remark}
\newtheorem{solution}[theorem]{Solution}
\newtheorem{summary}[theorem]{Summary}
\newtheorem*{thm}{Theorem}

\def \bp{{\mathbf p}}
\def \bA{{\mathbf A}}
\def \bL{{\mathbf L}}
\def \bF{{\mathbf F}}
\def \bS{{\mathbf S}}
\def \bC{{\mathbf C}}
\def \bn{{\mathbf n}}
\def \bm{{\mathbf m}}
\def \scrA{{\mathscr A}}
\def \scrB{{\mathscr B}}
\def \scrC{{\mathscr C}}
\def \scrF{{\mathscr F}}
\def \Z{{\Bbb Z}}
\def \F{{\Bbb F}}
\def \D{{\Bbb D}}
\def \C{{\Bbb C}}
\def \N{{\Bbb N}}
\def \Q{{\Bbb Q}}
\def \G{{\Bbb G}}
\def \P{{\Bbb P}}
\def \K{{\Bbb K}}
\def \E{{\Bbb E}}
\def \A{{\Bbb A}}
\def \L{{\Bbb L}}
\def \BH{{\Bbb H}}
\def \T{{\Bbb T}}
\newcommand {\lu}[1]{\textcolor{red}{$\clubsuit$: #1}}
\newcommand {\zhu}[1]{\textcolor{red}{$\spadesuit$: #1}}

\title[Singularity categories of representations of algebras over local rings]{Singularity categories of representations of algebras over local rings}

\author[Lu]{Ming Lu}
\address{Department of Mathematics, Sichuan University, Chengdu 610064, P.R.China}
\email{luming@scu.edu.cn}

\subjclass[2000]{16E45, 18E30, 18E35}
\keywords{Representations of quivers over local rings; Gorenstein algebras; Orbit categories; Singularity categories, CM-finite}

\begin{abstract}
Let $\Lambda$ be a finite-dimensional algebra with finite global dimension, $R_k=K[X]/(X^k)$ be the $\Z$-graded local ring with $k\geq1$, and $\Lambda_k=\Lambda\otimes_K R_k$. We consider the singularity category $\cd_{sg}(\mod^\Z(\Lambda_k))$ of the graded modules over $\Lambda_k$. It is showed that there is a tilting object in $\cd_{sg}(\mod^\Z(\Lambda_k))$ such that its endomorphism algebra is isomorphic to the triangular matrix algebra $T_{k-1}(\Lambda)$ with coefficients in $\Lambda$ and there is a triangulated equivalence between
$\cd_{sg}(\mod^{\Z/k\Z}(\Lambda))$ and the root category of $T_{k-1}(\Lambda)$. Finally, a classification of $\Lambda_k$ up to the Cohen-Macaulay representation type is given.

\end{abstract}

\maketitle

\section{Introduction}

The singularity category of an algebra is defined to be the Verdier quotient of the bounded derived category with respect to the thick subcategory formed by complexes isomorphic to bounded complexes of finitely generated projective modules \cite{Bu}, see also \cite{Ha3}. Recently, D. Orlov's global version \cite{Or1} attracted a lot of interest in algebraic geometry and theoretical physics. In particular, the singularity category measures the homological singularity of an algebra \cite{Ha3}: the algebra has finite global dimension if and only if its singularity category is trivial.

Singularity categories have deep relationship with Gorenstein algebras.
A fundamental result of R. Buchweitz \cite{Bu} and D. Happel \cite{Ha3} states that for a Gorenstein algebra $A$, its singularity category is triangulated equivalent to the stable category of Gorenstein projective (also called (maximal) Cohen-Macaulay) $A$-modules, which generalizes J. Rickard's result \cite{Ri} on self-injective algebras.
For any artin algebra $A$, denote by $\Gproj(A)$ its subcategory of Gorenstein projective modules.
Inspired by the representation type of algebras, one can define Cohen-Macaulay representation type of algebras, explicitly, if $\Gproj(A)$ has only finitely many non-isomorphic indecomposable objects, then $A$ is called to be of finite Cohen-Macaulay type or CM-finite, see e.g. \cite{LiZ,Be2}; otherwise, $A$ is called to be of infinite Cohen-Macaulay type or CM-infinite.

Another motivation of this paper is the root category of a finite-dimensional algebra. Root category was first introduced by D. Happel \cite{Ha1} for finite-dimensional hereditary
algebras. Let $A$ be a finite-dimensional hereditary algebra over a field $K$. Let $\cd^b(\mod(A))$ be the
derived category of finitely generated right $A$-modules. Then the root category $\mcr_A$ of $A$ is
defined to be the $2$-periodic orbit category $\cd^b(\mod(A))/\Sigma^2$, where $\Sigma$ is the shift functor. It was proved by L. Peng and J. Xiao \cite{PX} that the root category $\mcr_A$ is a triangulated category. With this triangle structure,
L. Peng and J. Xiao \cite{PX2} constructed a so called Ringel-Hall Lie algebra associated to each root
category and realized all the symmetrizable Kac-Moody Lie algebras. In general, for $A$ not hereditary, $\cd^b(\mod(A))/\Sigma^2$ is not triangulated, with the help of the DG orbit categories defined in \cite{Ke}, one can construct another
$2$-periodic triangulated category which is a triangulated hull of $\cd^b(\mod(A))/\Sigma^2$. This triangulated hull is also called the root category of $A$ and is denoted also by $\mcr_A$, see e.g. \cite{Fu}.

Let $\Lambda$ be a finite-dimensional algebra with finite global dimension, $R_k=K[X]/(X^k)$ be the $\Z$-graded local ring with $X$ degree $1$ where $k$ is a positive integer, and $\Lambda_k=\Lambda\otimes_K R_k$. Then $\Lambda_k$ is a positively graded Gorenstein algebra. C. M. Ringel and M. Schmidmeier \cite{RS} investigate the Gorenstein projective modules over $\Lambda_k$ with $\Lambda$ hereditary of type $\A_2$ by using submodule categories, and describe its Auslander-Reiten quiver for $k<6$. The structure of this kind of submodule categories has been studied by many people, see \cite{RS0,RS,Sim,Sim2,Sim3} and the references therein, and have been discovered to be related to weighted projective lines, see \cite{KLM0,KLM,Len} and the references therein.

Later, C. M. Ringel and P. Zhang \cite{RZ} prove that the singularity category of $\Lambda_2$ is triangulated equivalent to the triangulated orbit category $\cd^b(\mod(\Lambda))/\Sigma$ for $\Lambda$ a path algebra $KQ$, where $Q$ is an acyclic quiver.
X.-H. Luo and P. Zhang generalize these works and introduce monic representations to describe the Gorenstein projective modules over $A\otimes_K KQ$ (also $A\otimes_K (KQ/I)$ with $I$ generated by monomial relations) with $A$ a finite-dimensional algebra \cite{LuoZ1,LuoZ2}. Recently, D. Shen describes Gorenstein projective modules over the tensor product of two algebras in terms of their underlying one-sided modules \cite{Shen}, see also \cite{HLXZ}.

In this paper, we mainly consider the singularity categories $\cd_{sg}(\mod(\Lambda_k))$ and $\cd_{sg}(\mod^\Z(\Lambda_k))$ over $\Lambda_k$.

K. Yamaura \cite{Ya} proved that for a positively graded self-injective algebra, its stable category of the $\Z$-graded modules admits a tilting object. Following him, we prove that the singularity category $\cd_{sg}(\mod^\Z(\Lambda_k))$ of the $\Z$-graded modules over $\Lambda_k$ has a tilting object with the same construction, in particular, its endomorphism algebra is isomorphic to the triangular matrix algebra $T_{k-1}(\Lambda)$ with coefficients in $\Lambda$. In this way, we get that $\cd_{sg}(\mod^\Z(\Lambda_k))$ is triangulated equivalent to the bounded derived category $\cd^b(\mod(T_{k-1}(\Lambda)))$. Viewing $\Lambda_k$ as a $\Z/k\Z$-graded algebra naturally and considering the singularity category $\cd_{sg}(\mod^{\Z/k\Z}(\Lambda_k))$, we prove that the above triangulated equivalence induces a triangulated equivalence $\cd_{sg}(\mod^{\Z/k\Z}(\Lambda_k))\simeq \mcr_{T_{k-1}(\Lambda)}$. Furthermore, when $k=2$, this result recovers the result of \cite{RZ}: $\cd_{sg}(\mod (\Lambda_2))\simeq\cd^b(\mod(\Lambda))/\Sigma$ if $\Lambda$ is hereditary. Finally, we classify $\Lambda_k$ up to the Cohen-Macaulay representation type, and give some examples to describe the Auslander-Reiten quivers for some $\Lambda_k$ of CM-finite type.

\vspace{0.2cm} \noindent{\bf Acknowledgments.}
The author deeply thanks Professor Bin Zhu for his guidance, inspiration, helpful discussion and comments.
The work was done during the stay of the author at the Department of Mathematics,
University of Bielefeld. He is deeply indebted to Professor Henning Krause for his kind
hospitality, inspiration and continuous encouragement. The author deeply thanks Professor Helmut Lenzing for helpful discussion on the tubular algebras.

The author deeply thanks the referee for his useful corrections and remarks.
The author was supported by the National Natural Science Foundation of China (No. 11401401 and No. 11601441).

\section{Preliminaries}
\subsection{Notations}

Throughout this paper $K$ is an algebraically closed field and algebras are finite-dimensional $K$-algebras unless specified. 

For a $K$-algebra $A$, we denote

$\triangleright$ $J_A$ --the Jacobson radical of $A$,

$\triangleright$ $A^{op}$ -- the opposite algebra,

$\triangleright$ $\Mod(A)$ -- category of the right $A$-modules,

$\triangleright$ $\mod(A)$ -- category of finitely generated right $A$-modules,

$\triangleright$ $\proj(A)$ -- category of finitely generated right projective $A$-modules,


$\triangleright$ $D=\Hom_K(-,K)$ -- the standard duality,

$\triangleright$ $\cd(\Mod(A))$ -- derived category of the right $A$-modules,

$\triangleright$ $\cd^b(\mod(A))$ -- bounded derived category of finitely generated right $A$-modules,


$\triangleright$ $\Sigma$ -- the shift functor,

$\triangleright$ $\gldim A$ -- global dimension of $A$.

We identify $\mod(A^{op})$ with the category of the finitely generated left $A$-modules.
For an additive category $\ca$, we use $\Ind\ca$ to denote the set of all non-isomorphic indecomposable objects in $\ca$.

\subsection{Group graded algebras}
\label{subsec:Group graded algebras}
Let $G$ be an abelian group with its multiplication denoted by $+$ and its identity element by $0$.
Let $A=\bigoplus_{i\in G} A_i$ be a $G$-graded algebra. A $G$-graded $A$-module $X$ is of form $\bigoplus_{i\in G}X_i$, where $X_i$ is the degree $i$ part of $X$.
The category $\mod^G(A)$ of (finitely generated) $G$-graded $A$-modules is defined as follows.
\begin{itemize}
\item The objects are $G$-graded $A$-modules,
\item For $G$-graded $A$-modules $X$ and $Y$, the morphism space from $X$ to $Y$ in $\mod^G(A)$ is defined by
$$\Hom_{\mod^G(A)}(X,Y):=\{ f\in \Hom_A(X,Y)| f(X_i)\subseteq Y_i\mbox{ for any }i\in G\}.$$
\end{itemize}
We denote by $\proj^G(A)$ the full subcategory of $\mod^G(A)$ consisting of projective objects.

For $i\in G$, the \emph{grade shift functor} $(i):\mod^\Z(A)\rightarrow \mod^\Z(A)$, $X \mapsto X(i)$, is defined by letting $X(i) =\oplus_{j\in G}X(i)_j$, where $X(i)_j:=X_{j+i}$.
Then for any two $G$-graded $A$-modules $X,Y$, we have
$$\Hom_A(X,Y)=\bigoplus_{i\in G}\Hom_{\mod^G(A)}(X,Y(i)).$$

It is a well-known fact that there is a forgetful functor $F:\mod^G(A)\rightarrow \mod(A)$, which associates to $M\in \mod^G(A)$ the underlying
ungraded $A$-module.

Recall that $D:\mod(A)\rightarrow\mod(A^{op})$ is the standard duality. It induces the following duality.
For any $X\in \mod^G(A)$, we regard $D(X)$ as a $G$-graded $A^{op}$-module by defining $(D(X))_i:=D(X_i)$ for any $i\in G$. Then we have the duality
$D:\mod^G(A)\rightarrow \mod^G(A^{op})$.

Under a suitable assumption of $A$, we give a description of simple objects, projective objects, injective objects in $\mod^G(A)$ up to isomorphisms.

\begin{proposition}[\cite{Ya}]\label{proposition characterize of simple porjective injective modules of graded algebras}
Assume that $J_A=J_{A_0}\oplus (\bigoplus_{i\in G\setminus\{0\}}A_i)$. We take a set $\overline{PI}$ of idempotents of $A_0$ such that $\{eA_0 \mid e\in\overline{PI}\}$ is a complete list of the non-isomorphic indecomposable projective $A_0$-modules. Then the following assertions hold.
\begin{itemize}
\item[(a)] Any complete set of orthogonal primitive idempotents of $A_0$ is that of $A$.
\item[(b)] A complete list of the non-isomorphic simple objects in $\mod^G(A)$ is given by
$$\{ S(i)\mid i\in G,S \mbox{ is a simple }A_0\mbox{-module}\}.$$
\item[(c)] A complete list of the non-isomorphic indecomposable projective objects in $\mod^G(A)$ is given by
$$\{e(i)A \mid i\in G,e\in\overline{PI}\}.$$
\item[(d)] A complete list of the non-isomorphic indecomposable injective objects in $\mod^G(A)$ is given by
$$\{ D( Ae)(i)\mid i\in G,e\in\overline{PI}\}.$$
\end{itemize}
\end{proposition}


Let $\overline{G}=G/H$ be a factor group of $G$, where $H$ is a subgroup of $G$. Denote by $\overline{i}:=i+H$ in $\overline{G}$ for any $i\in G$.
For any $G$-graded algebra $A$, it can be viewed as a $\overline{G}$-graded algebra, i.e., $$A_{\overline{i}}=\bigoplus_{j\in i+H}A_j,\,\,\forall\overline{i}\in \overline{G}.$$
Furthermore, in this case, any $G$-graded $A$-module $X$ can be viewed as a $\overline{G}$-graded module by letting
$$X_{\overline{i}}=\bigoplus_{j\in i+H}X_j,\,\,\forall \overline{i}\in \overline{G}.$$
Then there is an exact functor $F_{\overline{G}}:\mod^G(A)\rightarrow \mod^{\overline{G}}(A)$, which is also called the \emph{forgetful functor}. Let $X$ and $Y$ be $G$-graded $A$-modules. Then
$$\Hom_{\mod^{\overline{G}}(A)}(F_{\overline{G}}(X),F_{\overline{G}}(Y))=\bigoplus_{i\in H}\Hom_{\mod^G (A)}(X,Y(i)).$$

For the special case $H=G$, obviously, $\mod^{\overline{G}}(A)=\mod(A)$.

Finally we recall the tensor algebra of two $G$-graded algebras.

For any two $G$-graded algebras $A$ and $B$, we define a $G$-grading on the tensor algebra $A\otimes_K B$ by letting
\begin{align*}
(A\otimes_K B)_i :=\left\{
a \otimes b\mid\sum a\otimes b, a\in A_k,b\in B_l, k+l=i\right\}
\end{align*}
for any $i\in G$.

Moreover let $C$ be a $G$-graded algebra, $X$ be a $G$-graded $A^{op}\otimes_KB$-module and $Y$ be a $G$-graded $B^{op}\otimes_K C$-module. We define a $G$-grading on the $A^{op}\otimes_K C$-module $X\otimes_B Y$ by
\begin{align*}
(X \otimes_B Y)_i :=\left\{\sum x \otimes y \mid x \in X_k, y \in Y_l, k + l = i\right\}
\end{align*}
for any $i\in G$.

\subsection{Gorenstein algebra and Gorenstein projective modules}

\begin{definition}
A finite-dimensional $K$-algebra $A$ is called a Gorenstein (or Iwanaga-Gorenstein) algebra if $A$ satisfies $\id\,_A A<\infty$ and $\id A_A <\infty$.
\end{definition}
Then a $K$-algebra $A$ is Gorenstein if and only if $\id A_A<\infty$ and $\pd D({}_AA)<\infty$.
For a Gorenstein algebra $A$, by Zask's Lemma we have $\id{} _AA=\id A_A$, see also \cite[Lemma 6.9]{AR} and \cite[Lemma 1.2]{Ha3}, the common value is denoted by $\Gd A$. If $\Gd A\leq d$, we say that $A$ is a \emph{$d$-Gorenstein} algebra.

To illustrate the introduced notion we present below several
examples.
\begin{example}
\label{example 1}
(a) Let $A$ be an algebra of global dimension $d$. Then $A$ is a $d$-Gorenstein algebra.

(b) Let $Z_n$ be the quiver with $n$ vertices and $n$ arrows which forms an oriented cycle. The vertex set of $Z_n$ is $\{0,1,\dots,n-1\}$. Let $I_m$ be the two-sided ideal of $KZ_n$ generated by all paths of length $m$, for $m\ge2$. Then $A=KZ_n/I_m$ is a self-injective algebra, i.e. $0$-Gorenstein algebras.

(c) Retain the notations as in (b). Let $B$ be an algebra of global dimension $d$. Then $A\otimes_K B$ is a $d$-Gorenstein algebras, which is not self-injective in general.

(d) Let $A_i$  be $d_i$-Gorenstein algebras for $i=1,2$. Then $A_1\otimes_K A_2$ is a $(d_1+d_2)$-Gorenstein algebra, see e.g. \cite{AR,MT}.

(e)  Let $Q$ be a finite acyclic quiver and $\mathcal{C}_Q$ the associated cluster category. Let $T$ be a tilting module of the path algebra $KQ$. Denote by $A={\rm End}_{\mathcal{C}_Q}(T)$ the cluster-tilted algebra. It infers from \cite{KR} that $A$ is $1$-Gorenstein, which usually has infinite global dimension.

(f) Let $A$ be a \emph{gentle algebra}. Then $A$ is a Gorenstein algebra, see \cite{GR}.
\end{example}

\begin{definition}[\cite{EJ}]
Let $A$ be a Gorenstein algebra. A finitely generated $A$-module $M$ is called to be Gorenstein projective (or (maximal) Cohen-Macaulay) if
$$\Ext^i_A(M,A)=0\mbox{ for }i\neq0.$$
\end{definition}

The full subcategory of Gorenstein projective modules in $\mod(A)$ is denoted by $\Gproj(A)$.

For a module $M$ take a short exact sequence $0\rightarrow\Omega(M)\rightarrow P\rightarrow M\rightarrow0$ with $P$ projective. The module $\Omega(M)$ is called a \emph{syzygy module} of $M$. Syzygy modules of $M$ are not uniquely determined, while they are naturally isomorphic to each other in the stable category $\underline{\mod}\, A$. For each $d \ge 1$ denote by $\Omega^d(\mod(A))$ the subcategory of modules of the form $\Omega^d(M)$ for an $A$-module $M$.

\begin{theorem}[\cite{EJ}]
   \label{theorem characterize of gorenstein property}
Let $A$ be an algebra and let $d\geq0$. Then the following statements are equivalent:
\begin{enumerate}
\item[(a)]
the algebra $A$ is $d$-Gorenstein;
\item[(b)]
$\Gproj(A)= \Omega^d(\mod(A))$.
\end{enumerate}
In this case, an $A$-module $G$ is Gorenstein projective if and only if there is an exact sequence $0\rightarrow G\rightarrow P^0\rightarrow P^1\rightarrow \cdots$ with each
$P^i$ projective.
\end{theorem}

\begin{example}
\label{example 2}
(a) Continue Example \ref{example 1} (a). The Gorenstein projective $A$-modules are just projective $A$-modules.

(b) Continue Example \ref{example 1} (b). All $A$-modules are Gorenstein projective.

(c) Continue Example \ref{example 1} (c). Note that all $A\otimes_K B$-modules are naturally right $A$ and $B$-modules. Then for an $A\otimes_K B$-module $X$, it is Gorenstein projective if and only if $X_B$ is projective, see \cite{Shen}.

(d) Continue Example \ref{example 1} (d). The $A_1\otimes_K A_2$-Gorenstein projective modules are described in \cite{Shen} in terms of their underlying onesided modules.

(e) Continue Example \ref{example 1} (e). The Gorenstein projective $A$-modules are just the torsionless modules.

(f) Continue Example \ref{example 1} (f). The Gorenstein projective $A$-modules are described in \cite{Ka}.
\end{example}

Similarly, for $G$-graded algebras, one can define the \emph{$G$-graded Gorenstein algebras}, \emph{$G$-graded Gorenstein projective modules}, see \cite{ALT}.
We denote by $\Gproj^G(A)$ the full subcategory of $\mod^G(A)$ formed by all $G$-graded Gorenstein projective modules.

For a finite-dimensional $G$-graded algebra $A$, let $F:\mod^G(A)\rightarrow \mod(A)$ be the forgetful functor, then for any $M\in\mod^G(A)$, $M$ is graded projective (resp. graded injective) if and only if $F(M)$ is projective (resp. injective) as $A$-module, see \cite[Corollary 2.1]{Na} or \cite[Proposition 1.3, Proposition 1.4]{GG}. In fact,
one has $G\text{-}\gid_A M=\id_A M$ for any graded $A$-module $M$ (see e.g. \cite[Theorem 2.1]{Na}), where $G\text{-}\gid_A M$ denotes the injective dimension in $\mod^G(A)$. 
Thus the $G$-graded algebra $A$ is $G$-graded $d$-Gorenstein if and only if it is $d$-Gorenstein as an ungraded ring (see e.g. \cite[Proposition 2.1]{ALT}).
It can be checked that all results concerning Gorenstein projective modules
in \cite{EJ} hold for $G$-graded Gorenstein projective modules (see e.g. \cite{ALT}).

\subsection{$\Z$-graded algebras}
Let $a$ be a nonnegative integer. In this paper, we mainly consider $\Z/a\Z$-graded algebras. Note that $\Z/a\Z$-graded algebras are just $\Z$-graded algebras if $a=0$.

From above, any $\Z$-graded algebra $A$ is $\Z/a\Z$-graded. Also we can regard $X\in\mod^\Z(A)$ as a $\Z/a\Z$-graded $A$-module.

Note that $\mod^{\Z/a\Z}(A)=\mod(A)$ if $a=1$.

The following example shows that every algebra can be viewed as a $\Z$-graded algebra.
\begin{example}
\label{example 5}
Let $A$ be a $K$-algebra. Then $A$ can be viewed as a $\Z$-graded algebra concentrated at degree zero. Similarly, for any $X\in\mod(A)$, it can be viewed as a $\Z$-graded $A$-module concentrated at degree zero. In this way, $\mod(A)$ is a subcategory of $\mod^\Z(A)$.
\end{example}

In the following, for a $\Z$-graded algebra $A$, by abuse of notations, we use $F:\mod^\Z(A)\rightarrow \mod(A)$ and $F:\mod^{\Z/a\Z}(A)\rightarrow \mod(A)$ to denote the forgetful functors; by noting that $\Z/a\Z$ is a quotient group of $\Z$, we have the forgetful functor $F_a:\mod^\Z (A)\rightarrow \mod^{\Z/a\Z} (A)$.
\begin{lemma}[see e.g. \cite{ALT,LZ}]\label{lemma forgetful functor}
Let $A=\bigoplus_{i\in\Z} A_i$ be a $\Z$-graded algebra. Then the forgetful functors $F:\mod^\Z (A)\rightarrow \mod(A)$ and $F:\mod^{\Z/a\Z} (A)\rightarrow \mod(A)$ induces forgetful functors from $\Gproj^\Z (A)$ to $\Gproj (A)$ and $\Gproj^{\Z/a\Z} (A)$ to $\Gproj (A)$ respectively, which are also denoted by $F$.
\end{lemma}

\begin{lemma}\label{lemma forgetful functor 2}
Let $A=\bigoplus_{i\in\Z} A_i$ be a $\Z$-graded algebra. Then the functor $F_a:\mod^\Z (A)\rightarrow \mod^{\Z/a\Z} (A)$ induces a functor from $\Gproj^\Z(A)$ to $\Gproj^{\Z/a\Z}(A)$, which is also denoted by $F_a$.
\end{lemma}
\begin{proof}
For any $M\in \mod^\Z(A)$, Lemma \ref{lemma forgetful functor} implies that $F(M)\in \Gproj(A)$. Since $F_a(M)\in\mod^{\Z/a\Z}(A)$ and $F(F_a(M))=F(M)\in \Gproj (A)$, we get that $F_a(M)\in \Gproj^{\Z/a\Z} (A)$ by Lemma \ref{lemma forgetful functor}.
\end{proof}

\begin{definition}[\cite{Or1}]
Let $A=\bigoplus_{i\in\Z}A_i$ be a $\Z$-graded algebra.
The $G$-graded singularity category is
$$\cd_{sg}(\mod^G(A)):=\cd^b(\mod^G(A))/\ck^b(\proj^G(A)),$$
where $G=\Z/a\Z$ for some integer $a\ge0$.
\end{definition}
In fact, we have the ungraded singularity category for the case $a=1$. More explicitly, the (ungraded) singularity category of $A$ is $\cd_{sg}(\mod(A)):=\cd^b(\mod(A))/\ck^b(\proj A)$.

Denote by $\pi:\cd^b(\mod(A))\rightarrow \cd_{sg}(\mod(A))$ and $\pi^G:\cd^b(\mod^G(A))\rightarrow \cd_{sg}(\mod ^GA)$
the localization functors for $G=\Z/a\Z$ with $a\geq0$.

Finally, we recall the Buchweitz-Happel's Theorem.
\begin{theorem}[\cite{Bu,Ha2,Ha3}]
Let $A$ be a $\Z$-graded Gorenstein algebra. Then $\Gproj^G(A)$ for any group $G=\Z/a\Z$ are Frobenius categories with the $G$-graded projective modules as the projective-injective objects. Furthermore, its stable category $\underline{\Gproj}^G(A)$ is triangulated equivalent to $\cd_{sg}(\mod^G(A))$.
\end{theorem}
By the way, in the above theorem, we have the ungraded
case for $a = 1$, and the $\Z$-graded case for $a=0$.

\subsection{Singularity categories of positively graded algebras}

A $\Z$-graded algebra is called \emph{positively graded} (or \emph{non-negatively graded}) if $A=\bigoplus_{i\geq0} A_i$.
Note that for a positively graded algebra $A$, the equation
$$J_A=J_{A_0}\oplus (\bigoplus_{i\neq0}A_i)=J_{A_{\bar{0}}}\oplus (\bigoplus_{\bar{i}\neq0} A_{\bar{i}})$$
always holds. So $A$ with both gradings above satisfies the assumption of Proposition \ref{proposition characterize of simple porjective injective modules of graded algebras}, see \cite[Proposition 2.18]{Ya}.

\begin{example}
\label{example 6}
(a) Let $Q$ be an arbitrary quiver, and $KQ$ be its path algebra. Then $KQ=\bigoplus_{i\geq0}KQ_i$, where for any $i\geq0$, $KQ_i$ is the subspace of $KQ$ generated by the set $Q_i$ of all paths of length $i$. Then it is well-known that
$KQ$ is a positively graded algebra.

(b) Retain the notations as in (a). Let $I$ be a homogeneous ideal of $KQ$, and $A:=KQ/I$. Then $A$ is also positively graded.
\end{example}

In the following, we always assume that $A$ is a positively graded algebra. Certainly, for a finite-dimensional positively graded algebra $A=\bigoplus_{i\geq0} A_i$ which we mainly focus on, $A$ is $d$-graded Gorenstein if and only if $A$ is $d$-Gorenstein. So we do not distinguish them in the following.

The additive functor $F_a:\mod^\Z(A)\rightarrow \mod^{\Z/a\Z}(A)$ induces an additive functor
$$\cd^b(F_a):\cd^b(\mod^\Z(A))\rightarrow \cd^b(\mod^{\Z/a\Z}(A)),$$
which is a triangulated functor. Similar to \cite[Proposition 2.25]{Ya}, we get the following lemma.

\begin{lemma}\label{lemma thick subcategory of derived categories}
We have $\thick(\Im \cd^b(F_a))=\cd^b(\mod^{\Z/a\Z}(A))$.
\end{lemma}
\begin{proof}
The proof is similar to that of \cite[Proposition 2.6]{Ya}.
Obviously, $\Im \cd^b(F_a)$ contains all simple objects in $\mod^{\Z/a\Z}(A)$. For any $M\in\mod^{\Z/a\Z}(A)$, it has a finite filtration by simple objects in $\mod^{\Z/a\Z}(A)$. Since short exact sequences in $\mod^{\Z/a\Z}(A)$ gives rise to triangles in $\cd^b(\mod^{\Z/a\Z}(A))$, $\thick(\Im \cd^b(F_a))$ contains all objects in $\mod^{\Z/a\Z}(A)$, and then $\thick(\Im \cd^b(F_a))=\cd^b(\mod^{\Z/a\Z}(A))$.
\end{proof}

For any $i>0$, note that $\Z/a\Z$ is a quotient group of $\Z/ia\Z$. Then we have the forgetful functor from $\mod^{\Z/ia\Z}(A)\longrightarrow\mod^{\Z/a\Z}(A)$, denoted by $F_a^{ia}$. Obviously, $F_a= F_{a}^{ia} \circ F_{ia}$.
Similarly, $F_a^{ia}$ induces a triangulated functor $\cd^b(F_a^{ia}):\cd^b(\mod^{\Z/ia\Z}(A))\longrightarrow\cd^b(\mod^{\Z/a\Z}(A))$.

Since all the exact functors $F_{a}$ and $F_a^{ia}$ preserve projective objects, they induce the triangulated functors $\underline{F_{a}}: \cd_{sg}(\mod^\Z(A))\rightarrow \cd_{sg}(\mod^{\Z/a\Z}(A))$, and $\underline{F_a^{ia}}: \cd_{sg}(\mod^{\Z/ia \Z}A)\rightarrow \cd_{sg}(\mod^{\Z/a\Z}(A))$ respectively.
In particular, we obtain the following commutative diagram:
\[ \xymatrix{ \cd^b(\mod^{\Z}(A)) \ar[rr]^{\pi}  \ar[d]^{\cd^b(F_{ia})} \ar@/_4pc/[dd]_{\cd^b( F_a)}&&  \cd_{sg}(\mod^\Z(A))  \ar[d]^{\underline{F_{ia}}} \ar@/^4pc/[dd]^{\underline{F_a}}\\
\cd^b(\mod^{\Z/ia\Z}(A)) \ar[rr]^{\pi^{\Z/ia\Z} } \ar[d]^{\cd^b(F_a^{ia})}&& \cd_{sg}(\mod^{\Z/ia\Z}(A)) \ar[d]^{\underline{F_a^{ia}}} \\
\cd^b(\mod^{\Z/a\Z}(A)) \ar[rr]^{\pi^{\Z/a\Z}} && \cd_{sg}(\mod^{\Z/a\Z}(A)).
 } \]

Similar to the proof of Lemma \ref{lemma forgetful functor 2}, we can prove that both $F_a$ and $F_a^{ia}$ preserve Gorenstein projective modules, and then they induce triangulated functors on the stable categories. By abuse of notations, we also denote them by $\underline{F_a}:\underline{\Gproj}^\Z (A)\rightarrow \underline{\Gproj}^{\Z/a\Z}(A)$ and $\underline{F_a^{ia}}:\underline{\Gproj}^{\Z/ia\Z} (A) \rightarrow \underline{\Gproj}^{\Z/a\Z}(A)$ respectively.

\begin{lemma}\label{lemma thick subcategory of singularity categories}
Let $A$ be a positively graded Gorenstein algebra. Then we have the following
\begin{itemize}
\item[(a)]
$\underline{\Gproj}^{\Z/a\Z} (A)= \thick(\Im \underline{F_{a}})=\cd_{sg}(\mod^{\Z/a\Z}(A))$,
 \item[(b)]$\thick(\Im \cd^b (F_{a}^{ia}))=\cd^b(\mod^{\Z/a\Z}(A))$,
 \item[(c)] $\underline{\Gproj}^{\Z/a\Z} (A)=\thick( \Im\underline{F_{a}^{ia}})=\cd_{sg}(\mod^{\Z/a\Z}(A))$.
 \end{itemize}
\end{lemma}
\begin{proof}
Since the functors $\pi,\pi^{\Z/a\Z}$ and $\pi^{\Z/ia\Z}$ are dense, all assertions follow from Lemma \ref{lemma thick subcategory of derived categories} and the above commutative diagram.
\end{proof}
\subsection{Existence of tilting objects in singularity categories}
First, we recall the definition of tilting objects.
Let $\ct$ be a triangulated category. An object $U\in\ct$ is called \emph{tilting} if it satisfies the following conditions.
\begin{itemize}
\item[(T1)] $\Hom_\ct(U,\Sigma^iU)=0$ for any $i\neq0$.
\item[(T2)] $\ct=\thick_\ct(U)$.
\end{itemize}

Let $\ct$ be an algebraic triangulated Krull-Schmidt category. If $\ct$ has a tilting object $U$, then there exists a triangulated equivalence $\ct\simeq \ck^b(\proj( \End_\ct(U)))$, see \cite{Ke1}.

Return to our setting, let $A$ be a positively graded algebra.
Following \cite{Ya}, the truncation functors $$(-)_{\geq i}:\mod^\Z(A)\rightarrow\mod^\Z(A),\quad (-)_{\leq i}:\mod^\Z(A)\rightarrow \mod^\Z(A)$$
are defined as follows. For a $\Z$-graded $A$-module $X$, $X_{\geq i}$ is a $\Z$-graded sub $A$-module of $X$
defined by
$$(X_{\geq i})_j:=\left\{ \begin{array}{ccc} 0& \mbox{ if }j<i \\ X_j& \mbox{ if }j\geq i,\end{array} \right.$$
and $X_{\leq i}$ is a $\Z$-graded factor $A$-module $X/X_{\geq i+1}$ of $X$.

Now we define a $\Z$-graded $A$-module by
\begin{align}
T:=\bigoplus_{i\geq0} A(i)_{\leq0}.
\end{align}

For large enough $i$, obviously, $A(i)_{\leq0}$ is projective since $A$ is finite-dimensional. So we can regard $T$ as an object in $\cd_{sg}(\mod^\Z(A))$ (by noting that every projective modules are zero in $\cd_{sg}(\mod^\Z(A))$).

\begin{lemma}[\cite{LZ}]\label{theorem existence of tilting object}
Let $A$ be a positively graded Gorenstein algebra where $A_0$ has finite global dimension. If $T=\bigoplus_{i\geq0} A(i)_{\leq0}$ is a Gorenstein projective $A$-module, then $T$ is a tilting object in $\cd_{sg}(\mod^{\Z} (A))$.
\end{lemma}

In the following, we assume that $T=\bigoplus_{i\geq0} A(i)_{\leq0}$ is a Gorenstein projective $A$-module. Let
\begin{align}
\label{eqn: Gamma}
\Gamma:=\End_{\cd_{sg}(\mod^{\Z} (A))}(T).
\end{align}
In general, it might be difficult to compute $\Gamma$. However, this is more manageable for certain kind of algebras $A$ recalled in the following.

Let $A$ be a positively graded Gorenstein algebra, $\Mod^\Z (A)$ be the category of all  graded right $A$-modules. We say that $A$ has \emph{Gorenstein parameter} $\ell$ if $\soc A$ is contained in $A_\ell$.
The endomorphism ring of $T$ in the graded singularity category is just the one in graded module category, thanks to the following lemma.

\begin{lemma}[\cite{Ya,LZ}]
  \label{theorem endomorphism ring of gorenstein projective modules}
Let $A$ be a positively graded Gorenstein algebra  of Gorenstein parameter $\ell$.  Assume that $T=\bigoplus_{i\geq0} A(i)_{\leq0}$ is a Gorenstein projective $A$-module.
Take a decomposition $T=\underline{T}\oplus P$ in $\Mod^\Z (A)$ where $\underline{T}$ is a direct sum of all indecomposable non-projective direct summand of $T$. Then
\begin{itemize}
\item[(a)] $\underline{T}$ is finitely generated, and is isomorphic to $T$ in $\underline{\Gproj}^{\Z} A$;

\item[(b)] there exists an algebra isomorphism $\Gamma\cong \End_{\mod^\Z(A)}(\underline{T})$;

\item[(c)] if $A_0$ has finite global dimension, then so does $\Gamma$. In this case, we have $\cd_{sg}(\mod^\Z(A))\simeq \cd^b(\mod(\Gamma))$.
\end{itemize}
\end{lemma}

Similar to \cite[Section 3.2]{Ya}, let $A$ be a positively graded Gorenstein algebra  of Gorenstein parameter $\ell$. Then $\underline{T}=\bigoplus_{i=0}^{\ell-1} A(i)_{\leq0}$.
In convenience, we write $\underline{T}$ as the following matrix form:
\begin{align}
\label{eqn: T 1}
\underline{T}=
\left( \begin{array}{l} A(\ell-1)_{\leq0}
\\
A(\ell-2)_{\leq0}
\\
\vdots
\\
A(0)_{\leq0}
 \end{array}\right)=
\begin{array}{cccccc}  1-\ell&2-\ell&\cdots &-1&0 \\ \left( \begin{array}{c} A_0\\ \\  \\ \\ \\
\end{array}\right. & \begin{array}{c} A_1\\ A_0 \\ \\ \\ \\
\end{array}&\begin{array}{c}\cdots\\ \cdots \\ \ddots\\ \\ \\
\end{array}& \begin{array}{c} A_{\ell-2}\\ A_{\ell-3} \\ \vdots\\ A_0 \\ \\
\end{array}& \left.\begin{array}{c} A_{\ell-1}\\ A_{\ell-2}  \\ \vdots \\ A_1 \\ A_0
\end{array}\right).  \end{array}
\end{align}
Here the numbers $1-\ell$, $2-\ell$, $\cdots$, $-1$, $0$ denote the degrees.
 Then Lemma \ref{theorem endomorphism ring of gorenstein projective modules} (b) shows that there exists an algebra isomorphism
\begin{align}
\label{eqn: endo ring}
\Gamma\cong \left(\begin{array}{ccccccc} A_0&A_1&\cdots &A_{l-2} &A_{l-1} \\
&A_0& \cdots & A_{l-3} &A_{l-2}\\
&& \ddots &\vdots&\vdots\\
&&& A_0&A_1\\
&&&&A_0
\end{array}  \right),
\end{align}
and we identify them via this isomorphism in the following. Then the algebra $\Gamma$ acts on $\underline{T}$ from both sides by the matrix multiplication, and we can regard $\underline{T}$ as a $\Gamma^{op}\otimes_K\Gamma$-module. In particular, $\underline{T}$ is isomorphic to $\Gamma$ as a $\Gamma^{op}\otimes_K\Gamma$-module.

On the other hand, $\underline{T}$ is a $\Z$-graded $\Gamma^{op}\otimes_K A$-module with left $\Gamma$-module structure given by the matrix multiplication.

Similar to \cite[Proposition 3.14]{Ya}, the triangulated equivalence in Lemma \ref{theorem endomorphism ring of gorenstein projective modules} (b) is given by the composition
\begin{align}
\label{eqn: composition 1}
\Xi:\cd^b(\mod(\Gamma))\xrightarrow{-\otimes^\L_\Gamma \underline{T} }\cd^b(\mod^\Z(A)) \stackrel{\pi^\Z}{\longrightarrow} \cd_{sg}(\mod^\Z(A)).
\end{align}

\subsection{Deriving DG categories}

In this subsection, first we collect basic facts on DG categories and their derived categories from \cite{Ke1}. After that, we recall the definition of derived-orbit categories defined in \cite{Ya} which is a special case of DG orbit categories defined in \cite{Ke}.

\begin{definition}[\cite{Ke1}]
An additive category $\scrA$ is called a differential graded (DG) category if the following conditions are satisfied.

(a) For any $X,Y\in\scrA$, the morphism set $\Hom_\scrA(X,Y)$ is a $\Z$-graded abelian group
$$\Hom_\scrA(X,Y)=\bigoplus_{i\in\Z} \Hom_\scrA^i(X,Y)$$
such that $gf\in\Hom_\scrA^{i+j}(X,Z)$ for any $f\in \Hom_\scrA^i(X,Y)$ and $g\in\Hom_\scrA^j(Y,Z)$ where $X,Y,Z\in\scrA$ and $i,j\in\Z$.

(b) The morphism set is endowed with a differential $d$ (i.e. $d^2=0$) of degree $1$ such that the equation
$$d(gf)=(dg)f+(-1)^j g(df)$$
hold for any $f\in \Hom_\scrA(X,Y)$ and $g\in\Hom_\scrA^j(Y,Z)$ where $X,Y,Z\in\scrA$ and $j\in\Z$.
\end{definition}

For a DG category $\scrA$, $H^0(\scrA)$ is the additive category whose objects are the same as $\scrA$, and the morphism set from $X$ to $Y$ is $H^0(\Hom_\scrA(X,Y))$.

Let $\scrA$ and $\scrB$ be DG categories. An additive functor $\scrF:\scrA\rightarrow \scrB$ is called a \emph{DG functor} if it preserves the grading and commutes with differentials.

We give the following typical examples of DG categories.
\begin{example}
\label{example 3}
(a) Let $\mathcal{A}$ be an additive category. Then $\ca$ is a DG category such that its morphism sets are concentrated at degree $0$, i.e., $\Hom_\ca^0(X,Y)=\Hom_\ca(X,Y)$ for any $X,Y\in\ca$.

(b)
Let $\ca$ be an additive category and $\cc(\ca)$ the category of complexes over $\ca$ with the morphism sets and the differential defined as follows:
\begin{itemize}
\item
For any two complexes $L,M$ and an integer $n\in\Z$, we define $\Hom^n_{\cc(\ca)}(L,M)$ to be the set formed by the morphisms $f:L\rightarrow M$ of graded objects of degree $n$, and
$\Hom_{\cc(\ca)}(L,M)=\bigoplus_{n\in\Z} \Hom_{\cc(\ca)}^n(L,M)$.
\item The differential is the commutator
$$d(f)=d_M f-(-1)^nfd_L$$ for any $f\in \Hom_{\cc(\ca)}^n(L,M)$.
\end{itemize}
In this case, $H^0(\cc(\ca))$ coincides with the homotopy category $\ck(\ca)$ of complexes over $\ca$.

Let $\cc^b(\ca)$ be the full subcategory of $\cc(\ca)$ formed by all bounded complexes. Then we have $H^0(\cc^b(\ca))=\ck^b(\ca)$.
\end{example}

\begin{example}
\label{example 4}
Any $\Z$-graded algebra $A=\bigoplus_{i\in\Z} A_i$ can be viewed as a DG category (denoted also by $A$) with trivial differential. In particular, any algebra $A$ can be viewed as a DG category whose morphism space is concentrated at degree $0$.
\end{example}

Let $\scrA$ be a small DG $K$-category. A \emph{DG $\scrA$-module} is a DG functor $\scrA^{op}\rightarrow \cc(\Mod(K))$.
Then DG $\scrA$-modules form a DG category which is denoted by $\cc_{\dg}(\scrA)$.

We denote by $\ck_{\dg}(\scrA):=H^0(\cc_{\dg}(\scrA))$ the homotopy category of DG $\scrA$-modules. By formally inverting quasi-isomorphisms in $\ck_{\dg}(\scrA)$, we obtain the derived category $\cd_{\dg}(\scrA)$ of DG $\scrA$-modules. From \cite{Ke1}, $\ck_{\dg}(\scrA)$ and $\cd_{\dg}(\scrA)$ are triangulated categories.

The following example shows that the homotopy category and derived categories of algebras are covered by the above ones.
\begin{example}
Let $A$ be a $K$-algebra. Let $\cc(\Mod(A))$ be the DG category of complexes over $\Mod(A)$ as in Example \ref{example 3} (b). We regard $A$ as a DG category as in Example \ref{example 4}.
Then we have another DG category $\cc_{\dg}(A)$ of DG $A$-modules, which coincides with $\cc(\Mod(A))$. Furthermore, $\ck_{\dg}(A)=\ck(\Mod(A))$ and $\cd_{\dg}(A)=\cd(\Mod(A))$.
\end{example}

We call a DG functor $\scrA\rightarrow \cc_{\dg}(\scrA)$ sending $X$ to $\Hom_\scrA(-,X)$ the \emph{Yoneda embedding}. In this way, $\scrA$ is a DG full subcategory of $\cc_{\dg}(\scrA)$, and $H^0(\scrA)$ is a full subcategory of $H^0(\cc_{\dg}(\scrA))=\ck_{\dg}(\scrA)$, which is also of $\cd_{\dg}(\scrA)$.

\begin{definition}[\cite{Ke}]
Let $\scrA$ be a DG category. We call the full subcategory $\thick_{\cd_{\dg}(\scrA)}(H^0(\scrA))$ of $\cd_{\dg}(\scrA)$ the triangulated hull of $H^0(\scrA)$.
\end{definition}

\subsection{Derived-orbit categories}

\begin{definition}[\cite{Ke}]
Let $\scrA$ be a DG category, and $\scrF:\scrA\rightarrow \scrA$ a DG functor. A DG category $\scrA/\scrF^+$ is defined as follows.
\begin{itemize}
\item The objects are the same as $\scrA$.
\item For $X,Y\in\scrA$, the morphism set from $X$ to $Y$ is defined by
$$\Hom_{\scrA/ \scrF^+}(X,Y):=\bigoplus_{n=0}^\infty \Hom_\scrA(\scrF^nX,Y).$$
\item The differential of $\scrA/\scrF^+$ is induced from that of $\scrA$.
\end{itemize}

The DG orbit category $\scrA/\scrF$ is defined with the objects the same as $\scrA$, and the morphism set from $X$ to $Y$ is defined by
$$\Hom_{\scrA/\scrF}(X,Y):=\colim(\Hom_{\scrA/\scrF^+}(X,Y) \xrightarrow{\scrF} \Hom_{\scrA/\scrF^+}(X,\scrF Y)\xrightarrow{\scrF}\cdots ),$$
with the differential of $\scrA/\scrF$ induced from that of $\scrA/\scrF^+$.
\end{definition}

Let $\scrA$ be a DG category, $\scrF:\scrA\rightarrow \scrA$ a DG functor, and the induced functor $H^0(\scrF):H^0(\scrA)\rightarrow H^0(\scrA)$. It is easy to see that
$$H^0(\scrA/\scrF)\simeq H^0(\scrA)/H^0(\scrF),$$
and then the triangulated hull $\thick_{\cd_{\dg}(\scrA/\scrF)}(H^0(\scrA/\scrF))$ of $H^0(\scrA/\scrF)$ is a triangulated category which contains $H^0(\scrA)/H^0(\scrF)$ as a full subcategory, and is generated by $H^0(\scrA)/H^0(\scrF)$, see \cite[Proposition 5.8]{Ya}.

Let $\Lambda$ be an algebra of finite global dimension, $M$ a bounded complex of $\Lambda^{op}\otimes_K\Lambda$-modules. Let $\scrA:=\cc^b(\proj(\Lambda))$ be the DG category of bounded complexes over $\proj(\Lambda)$. Then $H^0(\scrA)=\ck^b(\proj(\Lambda))=\cd^b(\mod(\Lambda))$. Since $\Lambda^{op}\otimes_K\Lambda$ has finite global dimension, there exists a quasi-isomorphism $pM\rightarrow M$ with a bounded complex $pM$ of projective $\Lambda^{op}\otimes_K\Lambda$-modules. Consider the DG functor:
$$\scrF:=-\otimes_\Lambda pM :\scrA\rightarrow \scrA.$$
Then we have the induced functor
$$H^0(\scrF)=-\otimes^\L_\Lambda M:\cd^b(\mod(\Lambda))\rightarrow \cd^b(\mod(\Lambda)).$$
We call the triangulated hull
\begin{align}
\cd(\Lambda,M):=\thick_{\cd_{\dg}(\scrA/\scrF)} (H^0(\scrA/\scrF))
\end{align}
of $H^0(\scrA/\scrF)$ the {\rm derived-orbit category} of $(\Lambda,M)$.

Let $A$ be a positively graded Gorenstein algebra, and $\Gamma$ an algebra of finite global dimension. We regard $\Gamma$ as a $\Z$-graded algebra as in Example \ref{example 5}. Let $U\in \cd^b(\mod  ^\Z (\Gamma^{op} \otimes _KA))$, $N\in \cd^b(\mod (\Gamma^{op}\otimes_K \Gamma))$, and
$$\varphi:=-\otimes_\Gamma^\L N : \cd^b(\mod(\Gamma)) \rightarrow \cd^b(\mod(\Gamma)).$$
We consider the derived tensor functor
$$-\otimes_\Gamma^\L U: \cd^b(\mod(\Gamma)) \rightarrow \cd^b(\mod^\Z(A))$$
and the canonical functor
$\pi^\Z:\cd^b(\mod^\Z(A)) \rightarrow \cd_{sg}(\mod^\Z(A))$.
Composing them, we have a triangulated functor
$$\psi:=\pi^\Z\circ (-\otimes_\Gamma^\L U):\cd^b(\mod(\Gamma)) \rightarrow \cd_{sg}(\mod^\Z(A)).$$

\begin{theorem}(see \cite[Theorem A.20]{IO})\label{theorem IO}
We fix an integer $a$. Assume that there exists a triangle
$$P\rightarrow U(a)\rightarrow N\otimes_\Gamma^\L U\rightarrow \Sigma P$$
in $\cd^b(\mod^\Z(\Gamma^{op}\otimes_K A))$ such that $P$ belongs to $\ck^b(\proj^\Z(A))$ as an object in $\cd^b(\mod^\Z(A))$.
Then there exists an additive functor $\cd^b(\mod(\Gamma))/\varphi\rightarrow \cd_{sg}(\mod^\Z(A))/(a)$
and a triangulated functor $\widetilde{\psi}: \cd(\Gamma,N)\rightarrow \cd_{sg}(\mod^{\Z/a\Z}(A))$ which makes the diagram
\[\xymatrix{ \cd^b(\mod(\Gamma)) \ar[rr]^{\psi}\ar[d] && \cd_{sg}(\mod^\Z(A))\ar[d] \ar@/^4pc/[dd]^{\underline{F_a}} \\
\cd^b(\mod(\Gamma))/\varphi \ar[rr]\ar[d] &&  \cd_{sg}(\mod^\Z(A))/(a) \ar[d]\\
\cd(\Gamma,N) \ar[rr]^{\widetilde{\psi}} && \cd_{sg}(\mod^{\Z/a\Z}(A))  }\]
commutative.
\end{theorem}

\subsection{Root categories}
\begin{definition}[\cite{Ke,Fu,PX}]
Let $\Lambda$ be an algebra of finite global dimension. For $M=\Sigma^m\Lambda$, and $\scrF=-\otimes_\Lambda\Sigma^m\Lambda$, we call the derived-orbit category $\cd(\Lambda,\Sigma^m\Lambda)$ of $(\Lambda,\Sigma^m\Lambda)$ the $m$-root category of $\Lambda$, and denote it by $\mcr_{\Lambda,m}$.
\end{definition}
By definition,  the $m$-root category $\mcr_{\Lambda,m}$ is the triangulated hull of $\cd^b(\mod(\Lambda))/\Sigma^m$.
In particular, the $2$-root category is called the \emph{root category} and denote by $\mcr_\Lambda$ briefly.

It is worth noting that when $\Lambda$ is a hereditary algebra, Peng and Xiao \cite{PX} proved that the orbit category $\cd^b(\mod(\Lambda))/\Sigma^m$ is triangulated via the homotopy category of $m$-periodic complexes of projective $\Lambda$-modules for any $m>1$ (Ringel and Zhang \cite{RZ} proved for the case $m=1$). In fact, in this case, the $m$-root category $\mcr_{\Lambda,m}$ is triangulated equivalent to the orbit category $\cd^b(\mod(\Lambda))/\Sigma^m$, see \cite{PX,Ke}.

\section{Singularity categories of $\Lambda_k=\Lambda\otimes_K R_k$}

Let $\Lambda$ be a finite-dimensional algebra over the field $K$, $R_k=K[X]/(X^k)$ the truncated polynomial ring for $k\geq0$. Obviously, $R_k$ is self-injective. For the case $k=2$, we call $R_k$ {\em the algebra of dual numbers}.

Let
\begin{align}
\Lambda_k:=\Lambda\otimes_K R_k
\end{align}
Then $\Lambda_k$ is also a finite-dimensional algebra. Furthermore, any $\Lambda_k$-module is naturally a $\Lambda$-module (also a $R_k$-module).

We endow $R_k=K[X]/(X^k)$ to be a $\Z$-graded algebra with $X$ degree $1$. Then $\Lambda_k$ is a positively graded algebra, i.e., $\Lambda_k=\bigoplus_{i=0}^{k-1} (\Lambda_k)_i$, where
$(\Lambda_k)_i\cong \Lambda$ for $0\le i\le k-1$.
Similarly, $\Lambda_k$ is natural a $G$-graded algebra for any $G=\Z/a\Z$. As in Example \ref{example 1} (c), we have that $\Lambda_k$ is a $d$-Gorenstein algebra if the global dimension of $\Lambda$ is $d$.

\begin{assumption}
We always assume that $d:=\gldim\Lambda$ is finite.
\end{assumption}

For any $\Lambda_k$-module $M$, following Example \ref{example 2} (c), we have that $M$ is Gorensetin projective if $M_\Lambda$ is projective.

In this paper, we mainly consider the (graded) singularity category of $\Lambda_k$, and describe it by using derived-orbit categories.
 \subsection{A description of $\mod^\Z(\Lambda_k)$}
We give a description of $\mod^\Z(\Lambda_k)$ by using representations of quivers over rings.
Before that, we fix some notations for quivers following \cite{ASS}.

Let $Q=(Q_0,Q_1,s,t)$
be a quiver. A path of length $l$ from $a$ to $b$ is a sequence $(a\mid \alpha_1,\alpha_2,\dots,\alpha_l\mid b)$, where $\alpha_i\in Q_1$ for all $1\le i\le l$, and we have
$s(\alpha_1)=a$, $t(\alpha_i)=s(\alpha_{i+1})$ for each $1\le i<l$, and $t(\alpha_l)=b$. Such a path is denoted briefly by $\alpha_1\alpha_2\cdots \alpha_l$.
The path algebra $KQ$ of $Q$ is the $K$-algebra whose underlying $K$-linear space has as its basis the set of all paths $\alpha_1\alpha_2\cdots \alpha_l$ for $l\geq0$ in $Q$ and the product is defined by
\begin{align*}
(\alpha_1\cdots \alpha_l) (\beta_1\cdots \beta_n)=\begin{array}{cc} \delta_{t(\alpha_l),s(\beta_1)} \alpha_1\cdots \alpha_l\beta_1\cdots \beta_n,  \end{array}
\end{align*}
where $ \delta_{t(\alpha_l),s(\beta_1)}$ is the \emph{Kronecker delta}.

Let $I$ be an admissible ideal of $KQ$. Then $(Q,I)$ is called a \emph{bound quiver}, and $A=KQ/I$ the \emph{bound quiver algebra}.
For a  $K$-algebra $R$, we can consider the representations of $(Q,I)$ over $R$. Explicitly, a representation of $(Q,I)$ over $R$ is of form $X=(X_i,X_\alpha)_{i\in Q_0,\alpha\in Q_1}$ such that $X_i\in\mod(R)$, and $X_\alpha:X_{s(\alpha)}\rightarrow X_{t(\alpha)}$ is a morphism of $R$-modules which is bound by $I$. We call such a representation finite-dimensional if $\sum_{i\in Q_0} \dim_K X_i<\infty$. Let $\rep_R(Q,I)$ be the finite-dimensional representations of $(Q,I)$ over $R$. Then it is well known that
$\mod(A\otimes_K R)\simeq\rep_R(Q,I)$.  We identify these two categories in the following.

Let $Q_{\A_\infty^\infty}$ be the infinite quiver:
$$\cdots\xrightarrow{\alpha_{-2}} -1\xrightarrow{\alpha_{-1}} 0\xrightarrow{\alpha_0} 1\xrightarrow{\alpha_1} \cdots.$$
Let $\widetilde{R}_k:= KQ_{\A_\infty^\infty}/ I$, where $I:= \langle\alpha_i\cdots\alpha_{i+k-1} \mid i\in\Z\rangle$.
Denote by $\rep_\Lambda(Q_{\A_\infty^\infty},I)$ the category of finite-dimensional representations of $(Q_{\A_\infty^\infty},I)$ over $\Lambda$.

\begin{lemma}[\cite{Ga}]
For any $k>0$, the category $\mod^\Z(\Lambda_k)$ is equivalent to $\rep_\Lambda(Q_{\A_\infty^\infty},I)$.
\end{lemma}
From now on we will identify these two categories

For a bound quiver algebra $\Lambda$, $\Lambda_k$ is also a bound quiver algebra. We describe the associated bound quiver in the following example.
\begin{example}
Let $\Lambda=KQ^\circ/I^\circ$ be a bound quiver algebra with $(Q^{\circ},I^\circ)$ its bound quiver. Let $Q$ be the quiver obtained from $Q^\circ$ by adding one loop $\varepsilon_i$ for each vertex $i$ in $Q^{\circ}$.
For any positive integer $k$, define $A=KQ/I_k$, where $I_k$ is the ideal of $KQ$ defined by the following relations:
\begin{itemize}
\item[(H1)] For each vertex $i$ we have the \emph{nilpotent relation}
$$\varepsilon_i^k=0.$$
\item[(H2)] For each arrow $(\alpha:i\rightarrow j)\in Q^{\circ}$ we have the \emph{commutativity relation}
$$\varepsilon_i\alpha=\alpha \varepsilon_j.$$
\item[(H3)] $I^\circ\subset I_k$ by viewing $I^\circ$ to be in $KQ$ naturally.
\end{itemize}

From \cite{Le}, we get that $A\cong\Lambda_k=\Lambda\otimes_K R_k$. So $\Lambda_k$ is a finite-dimensional $K$-algebra and the representations of $\Lambda_k$ are nothing else than representations of $(Q^\circ,I^\circ)$ over the ground ring $R_k$.

This algebra is also considered in \cite{GLS} for $I^\circ=0$. In this case, for $k=1$, we have $\Lambda_1=\Lambda$, which is a hereditary algebra; for $k=2$, it is shown in \cite{RZ} that $\Lambda_2$ is a $1$-Gorenstein algebra, and that the stable category of $\Gproj(\Lambda_2)$ is triangulated equivalent to the orbit category $\cd^b(\mod(\Lambda))/\Sigma$;
for $Q^\circ$ a quiver of type $\A_2$, Ringel and Schmidmeier \cite{RS} studied the category $\Gproj(\Lambda_k)$ for any $k>0$.
\end{example}

\subsection{Existence of tilting objects}

Let $A$ be a positively graded Gorenstein algebra. Recall that $A$ has \emph{Gorenstein parameter} $\ell$ if $\soc A$ is contained in $A_\ell$.

\begin{lemma}
We have that
$\Lambda_k=\bigoplus_{i=0}^\infty (\Lambda_k)_i=\bigoplus_{i=0}^{k-1} (\Lambda_k)_i$ is a positively graded Gorenstein algebra of Gorenstein parameter $k-1$.
\end{lemma}
\begin{proof}
It follows from the definition of $\Lambda_k$.
\end{proof}

For any $m>0$, let $T_m(\Lambda)$ be the upper triangular $m\times m$ matrix algebra with coefficients in $\Lambda$, i.e.
\begin{align}
\label{def: Tm}
T_m(\Lambda)=\left(\begin{array}{cccccccc} \Lambda& \Lambda &\cdots&\Lambda\\
&\Lambda&\cdots& \Lambda\\
&&\ddots&\vdots\\
&&&\Lambda \end{array} \right)_{m\times m}.
\end{align}
It is well known that $T_m(\Lambda)$ is a finite-dimensional algebra of global dimension $d+1$ for $m>1$, and of global dimension $d$ for $m=1$.

Let
\begin{align}
 \underline{T}:=\bigoplus_{i=0}^{k-2}\Lambda_k(i)_{\leq0}
\end{align}
Following \eqref{eqn: T 1}, it is convenient to write
$\underline{T}$ as the following matrix form,
$$\underline{T}=\begin{array}{cc} \\
\left( \begin{array}{cc} \Lambda_k(k-2)_{\leq 0}\\ \Lambda_k(k-3)_{\leq 0}\\ \vdots\\ \Lambda_k(1)_{\leq 0}\\ \Lambda_k(0)_{\leq 0}
  \end{array} \right)\end{array}=
\begin{array}{cccccc}  2-k&3-k&\cdots &-1&0 \\ \left( \begin{array}{c} \Lambda\\  \\ \\ \\ \\
\end{array}\right. & \begin{array}{c} \Lambda\\ \Lambda \\ \\ \\ \\
\end{array}&\begin{array}{c}\cdots\\ \cdots \\ \ddots\\ \\ \\
\end{array}& \begin{array}{c} \Lambda\\ \Lambda \\ \vdots\\ \Lambda \\ \\
\end{array}& \left.\begin{array}{c} \Lambda\\ \Lambda  \\ \vdots \\ \Lambda \\ \Lambda
\end{array}\right)  \end{array},$$
where $2-k,3-k,\dots,-1,0$ denote the degrees. 
Let
\begin{align}
\Gamma_k:= \End_{\cd_{sg}(\mod^\Z(\Lambda_k))} (\underline{T}).
\end{align}

We have the following proposition.
\begin{proposition}\label{lemma tilting object }
Retain the notations and assumption as above.
Then $\underline{T}$ is a tilting object in $\cd_{sg}(\mod^\Z(\Lambda_k))$, and
$\Gamma_k\cong T_{k-1}(\Lambda)$.

In particular, the composition
\begin{align}
\label{dfn: Xi}
\Xi:=\pi^\Z\circ(-\otimes_{\Gamma_k}^\L \underline{T} ): \cd^b(\mod(\Gamma_k)) \rightarrow \cd_{sg}(\mod^\Z(\Lambda_k))
\end{align}
is a triangulated equivalence.
\end{proposition}
\begin{proof}
The proof is similar to that in \cite[Section 6.1]{Ya}.
Note that $\Lambda_k(i)_{\leq 0}$ is a projective $\Lambda_k$-module for any $i\geq k-1$ since $\Lambda_k$ has Gorenstein parameter $k-1$.
So $T=\bigoplus_{i=0}^\infty \Lambda_k(i)_{\leq0}$ is isomorphic to $\underline{T}=\bigoplus_{i=0}^{k-2}\Lambda_k(i)_{\leq0}$ in $\cd_{sg}(\mod^\Z(\Lambda_k))$.


On the other hand, one can check that there is an exact sequence
$$0\rightarrow (\Lambda_k(i)_{\leq0})(-k)\rightarrow \Lambda_k(-1)\rightarrow \Lambda_k(i)\rightarrow \Lambda_k(i)_{\leq0}\rightarrow0$$
for each $0\leq i\leq k-2$.
So $\Omega^2(\underline{T})\cong \underline{T}(-k)$ in the stable category $\underline{\mod}^\Z (\Lambda_k)$.
Since $\Lambda_k$ is a positively graded Gorenstein algebra, we have that $\underline{T}$ is a graded Gorenstein projective $\Lambda_k$-module by the graded version of Theorem \ref{theorem characterize of gorenstein property}.
Then Lemma \ref{theorem existence of tilting object} yields that $\underline{T}$ is a tilting object in $\cd_{sg}(\mod^\Z(\Lambda_k))$.

It follows from \eqref{eqn: endo ring} that $\Gamma_k\cong T_{k-1}(\Lambda)$.

The last assertion follows from Lemma \ref{theorem endomorphism ring of gorenstein projective modules} and \eqref{eqn: composition 1}.
\end{proof}

From now on, we identify $\Gamma_k$ with $T_{k-1}(\Lambda)$. Recall that $\underline{T}$ is a $\Gamma_k^{op}\otimes_K\Gamma_k$-module  with the left and right actions given by the matrix multiplication. 

\begin{remark}\label{remark 1}
For $N>0$, it is worth noting that O. Iyama, K. Kato and Jun-ichi Miyachi defined the homotopy category of $N$-complexes and the derived category of $N$-complexes in \cite{IKM}. The homotopy category of $N$-complexes of projective $\Lambda$-modules is triangulated equivalent to the homotopy category of projective $T_{N-1}(\Lambda)$-modules for any algebra $\Lambda$, see \cite{IKM,BHN}.

Using \cite[Proposition 1.2]{Shen}, one can prove that $\Gproj^\Z(\Lambda_k)$ is equivalent to the category of $k$-complex of projective $\Lambda_k$-modules as Frobenius categories, which yields that $\underline{\Gproj}^\Z(\Lambda_k)$ is triangulated equivalent to the homotopy category of $k$-complexes of projective $\Lambda_k$-modules, and so Proposition \ref{lemma tilting object } can be proved in this way.
\end{remark}

\subsection{Main result}
Let
\begin{align}
M=\bigoplus^{2k-2}_{i=k} \Lambda_k(i)_{\geq1-k}.
\end{align}
Similarly, it is convenient to write $M$ as the following matrix form
$$M=\begin{array}{cc} \\
\left( \begin{array}{cc} \Lambda_k(2k-2)_{\geq 1-k}\\ \Lambda_k(2k-1)_{\geq 1-k}\\ \vdots\\ \Lambda_k(1-k)_{\geq 1-k}\\ \Lambda_k(k)_{\geq 1-k}
  \end{array} \right)\end{array}=
\begin{array}{cccccc}  1-k&2-k&\cdots &-2&-1 \\ \left( \begin{array}{c} \Lambda\\ \Lambda \\ \vdots \\ \Lambda \\ \Lambda\\
\end{array}\right. & \begin{array}{c} \\ \Lambda \\ \vdots \\ \Lambda \\ \Lambda\\
\end{array}&\begin{array}{c}\\ \\ \ddots \\ \cdots\\ \cdots \\
\end{array}& \begin{array}{c} \\ \\  \\ \Lambda\\ \Lambda \\
\end{array}& \left.\begin{array}{c} \\\\\\\\ \Lambda\\
\end{array}\right)  \end{array}.$$

Recall that $\Gamma_k=T_{k-1}(\Lambda)$.
The algebra $\Gamma_k$ acts on $M$ from both sides by matrix multiplication, which implies that $M$ is a $\Gamma_k^{op}\otimes_K \Gamma_k$-module. Since the action of $\Gamma_k$ on $M$ from the left commutes with that of $\Lambda_k$ from the right, we can also regard $M$ as a $\Z$-graded $\Gamma^{op}_k\otimes \Lambda_k$-module.

\begin{lemma}\label{lemma existence of some triangles}

The following assertions hold.

(a) There exist isomorphisms $M\otimes_{\Gamma_k}^\L \underline{T}\cong  M$ and $\underline{T}\otimes_{\Gamma_k}^\L \underline{T}\cong \underline{T}$ in $\cd^b(\mod^\Z(\Gamma_k^{op}\otimes_K \Lambda_k))$.

(b) There exist two short exact sequences
$$0\rightarrow M\rightarrow \bigoplus_{i=k}^{2k-2} \Lambda_k(i) \rightarrow \underline{T}(k)\rightarrow0,\mbox{ and }0\rightarrow \underline{T}\rightarrow (\Lambda_k(k-1))^{\oplus (k-1) } \rightarrow M\rightarrow0$$
in $\mod^\Z(\Gamma_k^{op}\otimes \Lambda_k)$.
\end{lemma}
\begin{proof}
Similar to the proof of \cite[Lemma 6.3 and Lemma 6.4 (1)]{Ya}, we can prove (a) and get the first exact sequence in (b). So we only need to prove the existence of the second exact sequence in (b).

It is easy to see that it is an exact sequence in $\mod^\Z(\Lambda_k)$.
Similar to \eqref{eqn: T 1}, we can write $(\Lambda_k(k-1))^{\oplus (k-1) }$ in a matrix form, and the left $\Gamma_k$-module structure of $(\Lambda_k(k-1))^{\oplus (k-1) }$ is induced by the matrix multiplication.
Together with the left $\Gamma_k$-module structures of $\underline{T}$ and $M$, it is routine to check that
$$0\rightarrow \underline{T}\rightarrow (\Lambda_k(k-1))^{\oplus (k-1) } \rightarrow M\rightarrow0$$
is an exact sequence in $\mod^\Z(\Gamma_k^{op}\otimes \Lambda_k)$.
\end{proof}

Since each short exact sequence in $\mod^\Z(\Gamma_k^{op}\otimes \Lambda_k)$ gives a triangle in $\cd^b(\mod^\Z(\Gamma_k^{op}\otimes_K \Lambda_k ))$, we have
the following two triangles in $\cd^b(\mod^\Z(\Gamma_k^{op}\otimes_K \Lambda_k))$:
$$\bigoplus_{i=k}^{2k-2}\Lambda_k(i)\rightarrow \underline{T}(k)\rightarrow \Sigma M\otimes_{\Gamma_k}^\L \underline{T} \rightarrow \Sigma(\bigoplus_{i=k}^{2k-2} \Lambda_k(i))$$
and
$$\Sigma( (\Lambda_k(k-1))^{\oplus (k-1) })\rightarrow \Sigma M\rightarrow \Sigma^2\underline{T}\rightarrow \Sigma^2((\Lambda_k(k-1))^{\oplus (k-1) }),$$
by noting that $M\otimes_{\Gamma_k}^\L \underline{T}\cong M$ in $\cd^b(\mod^\Z(\Gamma_k^{op}\otimes_K \Lambda_k))$.
Combining them, we get the following commutative diagram by octahedral axiom:
\[\xymatrix{ \bigoplus_{i=k}^{2k-2} \Lambda_k(i)\ar[r] \ar@{=}[d]& V\ar[r] \ar[d] &   \Sigma( (\Lambda_k(k-1))^{\oplus (k-1) })\otimes_{\Gamma_k}^\L \underline{T} \ar[r]\ar[d]& \Sigma(\bigoplus_{i=k}^{2k-2} \Lambda_k(i))\ar@{=}[d]  \\
\bigoplus_{i=k}^{2k-2} \Lambda_k(i)\ar[r]  &\underline{T}(k) \ar[r]\ar[d] & \Sigma M\otimes_{\Gamma_k}^\L \underline{T}\ar[r] \ar[d] & \Sigma(\bigoplus_{i=k}^{2k-2} \Lambda_k(i))\\
& \Sigma^2\underline{T}\otimes_{\Gamma_k}^\L \underline{T}\ar@{=}[r] \ar[d] & \Sigma^2\underline{T}\otimes_{\Gamma_k}^\L \underline{T}\ar[d]&\\
&\Sigma V \ar[r] & \Sigma^2( (\Lambda_k(k-1))^{\oplus (k-1) })\otimes_{\Gamma_k}^\L \underline{T}.&
  }\]
Then the second column  gives the following triangle in $\cd^b(\mod^\Z(\Gamma_k^{op}\otimes_K \Lambda_k))$:
\begin{equation} \label{equation triangle 1}
V\stackrel{\alpha}{\longrightarrow} \underline{T}(k)\stackrel{\beta}{\longrightarrow} \Sigma^2 \underline{T}\otimes_{\Gamma_k}^\L \underline{T}\stackrel{\gamma}{\longrightarrow}\Sigma V.
\end{equation}
Similar to Lemma \ref{lemma existence of some triangles} (a), we obtain that $(\Lambda_k(k-1))^{\oplus (k-1) }\otimes_{\Gamma_k}^\L \underline{T}\cong  (\Lambda_k(k-1))^{\oplus (k-1) }$ in  $\cd^b(\mod^\Z(\Gamma_k^{op}\otimes_K \Lambda_k))$. In particular, $(\Lambda_k(k-1))^{\oplus (k-1) }\otimes_{\Gamma_k}^\L \underline{T}\in \ck^b(\proj^\Z(\Lambda_k))$. From the triangle in the first row (viewed in $\cd^b(\mod^\Z(\Lambda_k))$), it follows that $V\in \ck^b(\proj^\Z(\Lambda_k))$ by viewing $\ck^b(\proj^\Z(\Lambda_k))$ to be the thick subcategory of $\cd^b(\mod^\Z(\Lambda_k))$.

\begin{proposition}\label{proposition degree k equals shift 2}
We have the following commutative diagram of triangulated equivalences up to an isomorphism of functors.
\[\xymatrix{   \cd^b(\mod(\Gamma_k)) \ar[rr]^{\Xi\quad\quad} \ar[d]^{-\otimes_{\Gamma_k}^\L \Sigma^2\underline{T}} && \cd_{sg}(\mod^\Z(\Lambda_k)) \ar[d]^{(k)} \\
\cd^b(\mod(\Gamma_k)) \ar[rr]^{\Xi\quad\quad} && \cd_{sg}(\mod^\Z(\Lambda_k)).
}\]
In particular, $-\otimes_{\Gamma_k}^\L \Sigma^2\underline{T} \simeq\Sigma^2$ and there is an isomorphism of functors $ (k)\circ \Xi\stackrel{\sim}{\longrightarrow} \Xi\circ \Sigma^2$.
\end{proposition}
\begin{proof}
For convenience, denote by $\phi:=-\otimes_{\Gamma_k}^\L \Sigma^2\underline{T}$.
For $X\in \cd^b(\mod(\Gamma_k))$, we have a triangle
\begin{align}
\label{eqn: triangle 1}
\pi^\Z(X\otimes_{\Gamma_k}^\L V)\rightarrow \pi^\Z(X\otimes_{\Gamma_k}^\L \underline{T}(k)) \xrightarrow{\pi^\Z(X\otimes_{\Gamma_k}^\L \beta)} \pi^\Z(X\otimes_{\Gamma_k}^\L  \Sigma^2\underline{T}\otimes_{\Gamma_k}^\L \underline{T})\rightarrow\Sigma \pi^\Z(X\otimes_{\Gamma_k}^\L V)
\end{align}
in $\cd_{sg}(\mod^\Z(\Lambda_k))$ by applying $\pi^\Z\circ(X\otimes_{\Gamma_k}^\L-)$ to the triangle (\ref{equation triangle 1}). By definition, $\pi^\Z(X\otimes_{\Gamma_k}^\L \underline{T}(k))=\Xi(X)(k)$ and
$ \pi^\Z(X\otimes_{\Gamma_k}^\L \Sigma^2 \underline{T}\otimes_{\Gamma_k}^\L \underline{T})=\Xi \phi(X)$.
From the above analysis, we have $V\in \ck^b(\proj^\Z(\Lambda_k))$. Since $\gldim \Gamma_k<\infty$, we can assume that $X\in\ck^b(\proj(\Gamma_k))$. It follows that  $X\otimes_{\Gamma_k}^\L V\in \ck^b(\proj^\Z(\Lambda_k))$. From \eqref{eqn: triangle 1}, we have an isomorphism
$$ \pi^\Z(X\otimes_{\Gamma_k}^\L \beta): \Xi(X)(k)\rightarrow \Xi \phi(X)$$ in $\cd_{sg}(\mod^\Z(\Lambda_k))$.
Thus there exists an isomorphism of functors
\begin{align}
\label{eqn: equiv 1}
\pi^\Z(-\otimes_{\Gamma_k}^\L \beta): (k)\circ \Xi\rightarrow \Xi\circ \phi.
\end{align}

Recall that $\underline{T}$ is a $\Gamma^{op}_k\otimes_K \Gamma_k$-module with the $\Gamma$-action on both sides given by matrix multiplications, which is isomorphic to $\Gamma_k$ as a $\Gamma^{op}_k\otimes_K \Gamma_k$-module. So $\phi\simeq \Sigma^2$.

Together with \eqref{eqn: equiv 1}, we have the desired isomorphism of functors
$(k)\circ \Xi\stackrel{\sim}{\longrightarrow} \Xi\circ \Sigma^2$.
\end{proof}

\begin{corollary}\label{corolalary dense}
For any $n>0$, we have the following equivalences of additive categories:
$$\cd^b(\mod(\Gamma_k))/\Sigma^{2n}\simeq (\cd_{sg}(\mod^\Z(\Lambda_k)))/(kn)\simeq (\underline{\Gproj}^\Z(\Lambda_k))/(kn).$$
\end{corollary}
\begin{proof}
It follows from Proposition \ref{proposition degree k equals shift 2} that $(kn)\circ \Xi\stackrel{\sim}{\longrightarrow} \Xi\circ \Sigma^{2n}$ for any $n>0$. Then the equivalence $\Xi$ gives the desired equivalences of additive categories.
\end{proof}

Recall that $\mcr_{\Gamma_k,2n}= D(\Gamma_k,\Sigma^{2n}\Gamma_k)$ is the $2n$-root category of $\Gamma_k$. Then we have the following theorem.

\begin{theorem}
\label{main result}
For any integers $k>0$, $n>0$, there exists a triangulated equivalence
$$\widetilde{\Xi}: \mcr_{\Gamma_k,2n}\longrightarrow \cd_{sg} (\mod^{\Z/kn\Z}(\Lambda_k))$$
which makes the following diagram commutative:
\[\xymatrix{ \cd^b(\mod(\Gamma_k)) \ar[rr]^{\Xi\quad\quad}  \ar[d]^{\mathrm{nat.}} &&\cd_{sg}(\mod^\Z(\Lambda_k)) \ar[d]^{\underline{F_{kn}}} \ar[rr]^{\sim}&& \underline{\Gproj}^\Z (\Lambda_k) \ar[d]^{\underline{F_{kn}}} \\
\mcr_{\Gamma_k,2n} \ar[rr]^{\widetilde{\Xi}\quad\quad} && \cd_{sg}(\mod^{\Z/ kn\Z} (\Lambda_k)) \ar[rr]^{\sim} && \underline{\Gproj}^{\Z/kn\Z} (\Lambda_k).}\]
\end{theorem}
\begin{proof}
By using Proposition \ref{proposition degree k equals shift 2}, the proof is similar to \cite[Theorem 6.2]{Ya}, for convenience, we give it here.
By applying Theorem \ref{theorem IO} and Proposition \ref{proposition degree k equals shift 2}, we get that there exists a triangulated functor $\widetilde{\Xi}: \mcr_{\Gamma_k,2n}\rightarrow \cd_{sg} (\mod^{\Z/kn\Z}(\Lambda_k))$ such that the above diagram is commutative. By definition, $\mcr_{\Gamma_k,2n}$ is generated by $\cd^b(\mod(\Gamma_k))/\Sigma^{2n}$, and Lemma \ref{lemma thick subcategory of singularity categories} shows that $\cd_{sg}(\mod^{\Z/ kn\Z} (\Lambda_k))$ is generated by $(\cd_{sg}(\mod^\Z(\Lambda_k)))/(kn)$. Note that the restriction of $\widetilde{\Xi}$ to $\cd^b(\mod(\Gamma_k))/\Sigma^{2n}$ gives the equivalence of additive categories
$$\cd^b(\mod(\Gamma_k))/\Sigma^{2n}\simeq (\cd_{sg}(\mod^\Z(\Lambda_k)))/(kn)$$
in Corollary \ref{corolalary dense}. Therefore, $\widetilde{\Xi}$ is a triangulated equivalence.
\end{proof}

We have the following corollary directly.
\begin{corollary}
\label{cor: 1}
For any integer $k>0$, there exists a triangulated equivalence
$$\widetilde{\Xi}: \mcr_{\Gamma_k}\longrightarrow \cd_{sg} (\mod^{\Z/k\Z}(\Lambda_k))$$
which makes the following diagram commutative:
\[\xymatrix{ \cd^b(\mod(\Gamma_k)) \ar[rr]^{\Xi\quad\quad}  \ar[d]^{\mathrm{nat.}} &&\cd_{sg}(\mod^\Z(\Lambda_k)) \ar[d]^{\underline{F_k}} \ar[rr]^{\sim}&& \underline{\Gproj}^\Z(\Lambda_k) \ar[d]^{\underline{F_k}} \\
\mcr_{\Gamma_k} \ar[rr]^{\widetilde{\Xi}\quad\quad} && \cd_{sg}(\mod^{\Z/ k\Z} (\Lambda_k)) \ar[rr]^{\sim} && \underline{\Gproj}^{\Z/k\Z}(\Lambda_k). }\]
\end{corollary}

\subsection{Representations of algebras over the algebra of dual numbers}
Recall that an algebra $\Lambda$ is called {\em piecewise hereditary} if there exists a
hereditary abelian category $\ch$ such that the bounded derived categories $\cd^b(\mod(\Lambda))$ and $\cd^b(\ch)$ are
equivalent as triangulated categories, see \cite{HRS}. The following corollary is a generalization of \cite[Theorem 1]{RZ}.

\begin{theorem}
\label{thm: main 2}
We have a triangulated equivalence: $$\cd_{sg}(\mod^{\Z/n\Z}(\Lambda_2))\simeq \mcr_{\Lambda,n},\mbox{ for any }n>0.$$
Furthermore, if $\Lambda$ is piecewise hereditary, then
$$\cd_{sg}(\mod^{\Z/n\Z}(\Lambda_2))\simeq \cd^b(\mod(\Lambda))/\Sigma^n,\mbox{ for any }n>0.$$
\end{theorem}
\begin{proof}
For $k=2$, we get that $\underline{T}=\Lambda$, and $\Gamma_2=\Lambda$. Lemma \ref{lemma existence of some triangles} gives the exact sequence in $\mod^\Z((\Lambda)^{op}\otimes_K \Lambda_2)$:
 $$0\rightarrow \Lambda=\underline{T}\rightarrow \Lambda_2(1) \rightarrow \underline{T}(1)=\Lambda(1)\rightarrow0.$$
Then there is a triangle
$$ \Lambda_2(1)\rightarrow \underline{T}(1)\rightarrow \Sigma\underline{T}\rightarrow \Sigma\Lambda_2(1)$$
in $\cd^b(\mod^\Z(\Gamma_2^{op}\otimes_K \Lambda_2))$.
Similar to the proof of Proposition \ref{proposition degree k equals shift 2},
there exists a commutative diagram
\[\xymatrix{   \cd^b(\mod(\Lambda)) \ar[rr]^{\Xi\quad\quad} \ar[d]^{\Sigma} && \cd_{sg}(\mod^\Z(\Lambda_2)) \ar[d]^{(1)} \\
\cd^b(\mod(\Lambda)) \ar[rr]^{\Xi\quad\quad} && \cd_{sg}(\mod^\Z(\Lambda_2)),
}\]
i.e., $\Xi\circ\Sigma=(1)\circ \Xi$. Furthermore, for any $n>0$, $\Xi\circ\Sigma^n=(n)\circ \Xi$. Also similar to the proofs of Corollary \ref{corolalary dense} and Theorem \ref{main result}, we obtain that $\cd^b(\mod(\Lambda))/\Sigma^n\simeq \cd_{sg}(\mod^\Z(\Lambda_2))/(n)$ and there exists a triangulated equivalence
$$\widetilde{\Xi}: \mcr_{\Lambda,n}\longrightarrow \cd_{sg}(\mod^{\Z/n\Z}(\Lambda_2)),$$
for any $n>0$.

For the second assertion, if $\Lambda$ is piecewise hereditary, then the derived-orbit category $\mcr_{\Lambda,n}=D(\Lambda,\Sigma^{n}\Lambda)$ is the triangulated hull of $\cd^b(\mod(\Lambda))/\Sigma^n$ in $\cd(\cc_{\dg}^b(\proj \Lambda)/\Sigma^n)$. In this case, by the theorem in \cite[\S4]{Ke}, we get that $\cd^b(\mod(\Lambda))/\Sigma^n$ is already a triangulated category, which implies that $\mcr_{\Lambda,n}$ is triangulated equivalent to $\cd^b(\mod(\Lambda))/\Sigma^n$.
\end{proof}

\begin{corollary}[cf. \cite{RZ}]
\label{cor: RZ}
If $\Lambda$ is piecewise hereditary, we have $\underline{\Gproj}(\Lambda_2)\simeq\cd_{sg}(\mod(\Lambda_2))\simeq \cd^b(\mod(\Lambda))/\Sigma$.
\end{corollary}
\begin{proof}
It follows from Theorem \ref{thm: main 2} by noting that $\mod^{\Z/1\Z}(\Lambda_2)\simeq\mod(\Lambda_2)$.
\end{proof}

\subsection{For $\Gamma_k$ piecewise hereditary}
In this subsection, we describe $\cd_{sg}(\Lambda_k)$  for the case when $\Gamma_k$ is piecewise hereditary.

Recall the triangulated equivalence $\Xi= \pi^\Z\circ (-\otimes_{\Gamma_k}^\L \underline{T} )$ in \eqref{dfn: Xi}. Note that the grade shift functor $(1): \mod^\Z(\Lambda_k)\rightarrow \mod^\Z(\Lambda_k)$ gives a triangulated equivalence $(1):\cd_{sg}(\mod^\Z(\Lambda_k))\rightarrow \cd_{sg}(\mod^\Z(\Lambda_k))$. Then it induces a triangulated equivalence $\Psi: \cd^b(\mod(\Gamma_k))\rightarrow  \cd^b(\mod(\Gamma_k))$ which makes the following diagram commutative:
\[\xymatrix{   \cd^b(\mod(\Gamma_k)) \ar[rr]^{\Xi\quad\quad} \ar[d]^\Psi && \cd_{sg}(\mod^\Z(\Lambda_k)) \ar[d]^{(1)} \\
\cd^b(\mod(\Gamma_k)) \ar[rr]^{\Xi\quad\quad} && \cd_{sg}(\mod^\Z(\Lambda_k)).
}\]
In particular, $\Psi^k\simeq \Sigma^2$ by Proposition \ref{proposition degree k equals shift 2}.

\begin{lemma}
\label{lem: triangulated orbit}
Retain the notations as above. If $\Gamma_k$ is piecewise hereditary, then $\cd^b(\mod(\Gamma_k))/\Psi$ is a triangulated orbit category \`a la Keller \cite{Ke}.
\end{lemma}
\begin{proof}
We make the following observations.

First, since $\Gamma_k$ is piecewise hereditary, by definition, there is a hereditary abelian $K$-category $\ch$ such that
$\cd^b(\mod(\Gamma_k)) \stackrel{\sim}{\longrightarrow} \cd^b(\ch)$. We identify $\cd^b(\mod(\Gamma_k))$ with $\cd^b(\ch)$ in the following.

Second, for any indecomposable object $Y$ of $\cd^b(\ch)$, it follows from $\Psi^k\simeq \Sigma^2$ that only finitely many $\Psi^i(Y)$, $i\in\Z$, lie in $\ch$.

Finally, for the $\Psi$-orbit $\co_V$ of each indecomposable object $V$ in $\cd^b(\ch)$, by using $\Psi^k\simeq \Sigma^2$ again, there exists some indecomposable object $Y$ in $\ch$ such that $Y\in \co_V$ or $\Sigma Y\in \co_V$.

The assertion follows from these observations, by the theorem in \cite[\S4]{Ke}.
\end{proof}

\begin{proposition}
\label{prop: for piecewise hereditary}
If $\Gamma_k$ is piecewise hereditary, then we have the following triangulated equivalences
$$\cd^b(\mod(\Gamma_k))/\Psi\simeq \cd_{sg}(\mod(\Lambda_k))\simeq \underline{\Gproj}(\Lambda_k).$$
\end{proposition}
\begin{proof}
Similar to the proof of Corollary \ref{corolalary dense}, we have the following equivalences of additive categories
$\cd^b(\mod(\Gamma_k))/\Psi\simeq \cd_{sg}(\mod^\Z(\Lambda_k))/(1)$.
It follows from Lemma \ref{lem: triangulated orbit} that $\cd^b(\mod(\Gamma_k))/\Psi$ is a triangulated category. Then the orbit category $\cd_{sg}(\mod^\Z(\Lambda_k))/(1)$ is also a triangulated category. By Lemma \ref{lemma thick subcategory of singularity categories}, we get that $\cd_{sg}(\mod^\Z(\Lambda_k))/(1)\simeq \cd_{sg}(\mod(\Lambda_k))$ by noting that
$\mod^{\Z/1\Z}(\Lambda_k)\simeq \mod(\Lambda_k)$. It follows from $\cd^b(\mod(\Gamma_k))/\Psi\simeq \cd_{sg}(\mod^\Z(\Lambda_k))/(1)$ that the functor $\Xi$ induces the triangulated equivalence $\cd^b(\mod(\Gamma_k))/\Psi\simeq \cd_{sg}(\mod(\Lambda_k))$.
\end{proof}

\section{The Cohen-Macaulay type of $\Lambda_k$}

\subsection{CM-finite and graded CM-finite}
An algebra is of \emph{finite Cohen-Macaulay type}, or simply, \emph{CM-finite}, if there are only finitely many non-isomorphic indecomposable finitely generated Gorenstein projecitve modules. Otherwise, the algebra is called to be of \emph{infinite Cohen-Macaulay type}, or simply, \emph{CM-infinite}, see \cite{Be2,LZ}.

Let us return to consider group graded algebras.
For any group $G$, a finite-dimensional $G$-graded $K$-algebra $A$ is of {\em finite ($G$-)graded Cohen-Macaulay type} (or {\em $G$-graded CM-finite}) if there are only finitely many non-isomorphic indecomposable finitely generated Gorenstein projective modules up to the shifts of degree.

It is a well-known fact $A$ is CM-finite if and only if there are only finitely many iso-classes of indecomposable objects in $\underline{\Gproj}(A)$. Similar result also holds for $G$-graded algebras.

\begin{lemma}
\label{theorem equivalent graded CM classification to the ordinary}
A finite-dimensional $\Z$-graded $K$-algebra $A$ is graded CM-finite if $A$ is CM-finite.
\end{lemma}
\begin{proof}
From Lemma \ref{lemma forgetful functor 2}, we have the forgetful functor $F:\Gproj^\Z(A)\rightarrow \Gproj(A)$. Obviously, $\Im F=\Gproj^\Z(A)/(1)$.
Since $A$ is CM-finite, there are only finitely many indecomposable objects in $\Gproj^\Z(A)/(1)$ up to isomorphisms. Then $A$ is graded CM-finite by definition.
\end{proof}

Recall that every $\Z$-graded algebra can be viewed as $\Z/a\Z$-graded algebra. %
Then we have the following corollary directly.
\begin{corollary}\label{theorem equivalent graded CM classification to the ordinary 2}
Let $A$ be a finite-dimensional $\Z$-graded algebra. If $A$ is CM-finite, then $A$ is $\Z/a\Z$-graded CM-finite for any $a\geq 0$.
\end{corollary}

\subsection{Derived tubular algebras}

In the following, we classify $\Lambda_k$ up to Cohen-Macaulay representation type. Before that, we need recall some results about derived tubular algebras.

A {\em tubular algebra} is a tubular
extension of a tame concealed algebra of extension type $(2, 2, 2, 2)$, $(3, 3, 3)$, $(2, 4, 4)$,
or $(2, 3, 6)$ in the sense of \cite[page 230]{Rin}. Note that tubular algebras are piecewise hereditary \cite{GL}. 

We say that an algebra $A$ is {\em derived tubular} if it is
derived equivalent to a tubular algebra. It is reasonable to speak of the tubular type of a derived tubular algebra by the following result proved by D. Happel and C. M. Ringel  \cite{HR}: two tubular algebras are derived equivalent if and only if they have the same tubular type. H. Meltzer \cite{M} showed that each derived tubular algebra can be transformed by a finite sequence of APR-tilts and APR-cotilts to a
canonical algebra of tubular type. 

\begin{lemma}\label{lemma derived tubular algebras}
Let $\Lambda=KQ$ be a finite-dimensional hereditary algebra.
\begin{itemize}
\item[(i)] If $Q$ is of type $\D_4$, then $T_2(\Lambda)$ is a derived tubular algebra of type $(3,3,3)$;
\item[(ii)] If $Q$ is of type $\A_3$, then $T_3(\Lambda)$ is a derived tubular algebra of type $(2,4,4)$;
\item[(iii)] If $Q$ is of type $\A_5$, then $T_2(\Lambda)$ is a derived tubular algebra of type $(2,3,6)$;
\item[(iv)] If $Q$ is of type $\A_2$, then $T_5(\Lambda)$ is a derived tubular algebra of type $(2,3,6)$.
\end{itemize}
\end{lemma}
\begin{proof}
The assertions (ii)--(iv) follow from \cite[Theorem B]{KLM}.

For (i), from \cite[Theorem 2.1]{Ric3}, without loss of generality, we can assume that the bound quiver of $T_2(\Lambda)$ is as Figure 1 shows. Then we do APR-tilt at the vertex $7$, APR-cotilts at the vertices $4$, $1$, $3$ successively to obtain the algebra $B$ as Figure 2 shows. Then $T_2(\Lambda)$ is a derived tubular algebra since $B$ is a tubular algebra, see e.g. \cite{LP}.

\begin{center}\setlength{\unitlength}{0.6mm}
 \begin{picture}(110,50)(0,0)
 \put(-30,0){\begin{picture}(60,50)

\put(0,18){\vector(0,-1){16}}
\put(20,18){\vector(0,-1){16}}
\put(40,18){\vector(0,-1){16}}
\put(38,20){\vector(-1,0){16}}
\put(18,20){\vector(-1,0){16}}
\put(38,0){\vector(-1,0){16}}
\put(18,0){\vector(-1,0){16}}
\put(30,28){\vector(0,-1){16}}
\put(28.5,28.5){\vector(-1,-1){7}}
\put(28.5,8.5){\vector(-1,-1){7}}
\put(-1,-1){\tiny$7$}
\put(19,-1){\tiny$6$}
\put(39,-1){\tiny$8$}

\put(-1,18){$^3$}
\put(19,18){$^2$}
\put(39,18){$^4$}

\put(29,28){$^1$}
\put(29,8){$_5$}

\qbezier[12](18,18)(10,10)(2,2)
\qbezier[12](38,18)(30,10)(22,2)
\qbezier[12](29,28)(25,15)(21,2)

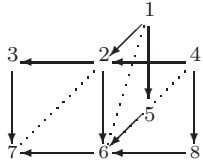
\put(-40,-28){Figure 1. The quiver of $T_2(\Lambda)$ for $\Lambda$ a}
\put(-12,-36){hereditary algebra of type $\D_4$.}

\end{picture}}

\put(60,-10){\begin{picture}(70,50)

\put(40,10){\tiny $8$}
\put(40,30){\tiny $2$}
\put(60,20){\tiny $4$}
\put(60,30){\tiny $7$}
\put(80,30){\tiny $6$}
\put(40,40){\tiny $3$}
\put(60,40){\tiny $1$}
\put(40,50){\tiny $5$}

\put(59,20){\vector(-2,-1){16}}
\put(79,30){\vector(-2,-1){16}}
\put(79,31){\vector(-1,0){16}}
\put(59,31){\vector(-1,0){16}}

\put(59,42){\vector(-2,1){16}}
\put(79,32){\vector(-2,1){16}}
\put(59,22){\vector(-2,1){16}}
\put(59,32){\vector(-2,1){16}}

\put(59,40){\vector(-2,-1){16}}

\put(48,17){\tiny$\gamma_3$}

\put(48,24){\tiny$\beta_3$}
\put(48,32){\tiny$\beta_2$}
\put(51,39){\tiny$\beta_1$}
\put(45,40){\tiny$\gamma_2$}

\put(48,48){\tiny$\gamma_1$}

\put(69,37){\tiny$\alpha_1$}

\put(69,32){\tiny$\alpha_2$}

\put(69,23){\tiny$\alpha_3$}

\put(40,0){$\beta_1\alpha_1+\beta_2\alpha_2=0$,}
\put(40,-8){$\beta_1\alpha_1+\beta_3\alpha_3=0$;}

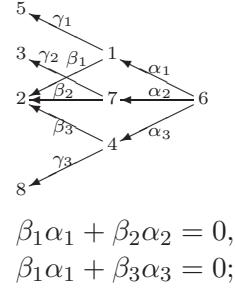
\put(10,-25){Figure 2. The bound quiver of $B$.}
\end{picture}}

\end{picture}
\vspace{1.5cm}

\end{center}

\end{proof}

In particular, for any derived tubular algebra $A$, $\cd^b(\mod(A))$ has infinitely many non-isomorphic indecomposable objects up to shifts of complexes, see \cite{HR}.

\subsection{CM-finiteness of $\Lambda_k$}


First, we give another proof of the following result.

\begin{lemma}[\cite{RS,Sim}, see also \cite{Sim2,Sim3}]
\label{lem: RSSIM}
If $\Lambda$ is a hereditary algebra of type $\A_2$, then $\Lambda_k=\Lambda\otimes_K R_k$ is CM-finite if and only if $k<6$.
\end{lemma}
\begin{proof}
For $k=1$, easily, $\Lambda_1=\Lambda$ is a hereditary algebra of type $\A_2$, which implies $\Lambda_1$ is CM-free.

For $k=2$, we get that $T_1(\Lambda)=\Lambda$ is a hereditary algebra of type $\A_2$. We obtain that $\Lambda_2$ is CM-finite since $\Lambda$ is representation-finite and $\underline{\Gproj}(\Lambda_2)\simeq \cd^b(\mod(\Lambda))/\Sigma$ by Corollary \ref{cor: RZ}.

For $k=3$, easily, $T_2(\Lambda)=KQ/I$ where $Q$ is the quiver as the following diagram shows:
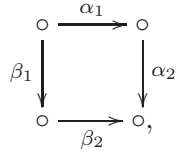
\[\xymatrix{\circ\ar[r]^{\alpha_1} \ar[d]_{\beta_1}& \circ\ar[d]^{\alpha_2} \\
\circ\ar[r]_{\beta_2} & \circ,}\]
 and $I=\langle\alpha_1\alpha_2-\beta_1\beta_2\rangle$.
Then $T_2(\Lambda)$ is a tilted algebra of type $\D_4$. It follows from Proposition \ref{prop: for piecewise hereditary} that $\underline{\Gproj}(\Lambda_3)\simeq \cd^b(\mod(T_2(\Lambda)))/\Psi$. Note that $\cd^b(\mod(T_2(\Lambda)))$ has only finitely many non-isomorphic indecomposable objects up to the shifts of complexes. Since $\Psi^k\simeq \Sigma^2$, it is easy to see that $\cd^b(\mod(T_2(\Lambda)))/\Psi$ has only fintely many non-isomorphic indecomposable objects. It follows that
$\Lambda_3$ is CM-finite.

For $k=4,5$, it is similar to the case $k=3$, only need to note that $T_3(\Lambda)$ and $T_4(\Lambda)$ are tilted algebras of type $\E_6$ and $\E_8$ respectively.

Since $\Lambda$ is of type $\A_2$, Lemma \ref{lemma derived tubular algebras} yields that $T_{5}(\Lambda)$ is a derived tubular algebra of type $(2,3,6)$. Similar to the proof of Lemma \ref{lem: triangulated orbit}, we get that $\cd^b(\mod  T_{5}(\Lambda))/\Sigma^2$ is already a triangulated category. Then
$\mcr_{T_{5}(\Lambda)}\simeq\cd^b(\mod  T_{5}(\Lambda))/\Sigma^2$.
For $\Lambda_6$, we get that $\underline{\Gproj}^{\Z/6\Z}(\Lambda_6)\simeq \mcr_{T_{5}(\Lambda)}\simeq \cd^b(\mod  T_{5}(\Lambda))/\Sigma^2$ by Corollary \ref{cor: 1}. Since $T_{5}(\Lambda)$ is a derived tubular algebra of type $(2,3,6)$, from the structure of its derived category described in \cite{HR}, we have that there are infinitely many non-isomorphic indecomposable objects in $\cd^b(\mod(T_{5}(\Lambda)))$ up to the shifts of complexes. Then $\Lambda_6$ is $\Z/6\Z$-graded CM-infinite, and is CM-infinite by Corollary \ref{theorem equivalent graded CM classification to the ordinary 2}.

For $\Lambda_k$ with $k>6$, from \cite[Proposition 1]{LS2}, we know that $T_{k-1}(\Lambda)$ is a wild algebra. Since $\Ind (\mod(T_{k-1}(\Lambda)))$ is a subset of $\Ind (\cd^b(\mod (T_{k-1}(\Lambda)))/\Sigma^2)$ and $\underline{\Gproj}^{\Z/k\Z}(\Lambda_k)\simeq \mcr_{T_{k-1}(\Lambda)}$ by Corollary \ref{cor: 1}, it is easy to see that
$\Lambda_k$ is a $\Z/k\Z$-graded CM-infinite algebra. From Corollary \ref{theorem equivalent graded CM classification to the ordinary 2}, we get that $\Lambda_k$ is CM-infinite for $k>6$.
\end{proof}

In the above corollary, it is easy to see that the number $s(k)$ of non-isomorphic indecomposable objects in $\Gproj(\Lambda_k)$ is given by the formula
$$s(k)=2+2(k-1)\frac{6}{6-k},$$
see \cite{RS}.

\begin{lemma}[see e.g. \cite{BGS,Vo}]\label{lemma derived finite category}
For a finite-dimensional algebra $B$ over $k$, if $\cd^b(\mod(B))$ has only finitely many non-isomorphic indecomposable objects up to the shift of complexes, then there exists a hereditary algebra $H$ of finite representation type such that $\cd^b(\mod(B))\simeq \cd^b(\mod(H))$.
\end{lemma}

\begin{theorem}\label{proposition CM finite}
Let $\Lambda$ be a finite-dimensional indecomposable algebra with finite global dimension. Then for any integer $k>0$, $\Lambda_k=\Lambda\otimes R_k$ is CM-finite if and only if it belongs to one of the following cases:
\begin{itemize}
\item[(i)] $k=1$;
\item[(ii)] $k=2$, $\Lambda$ is derived equivalent to a hereditary algebra of type $\A,\D$, or $\E$;
\item[(iii)] $k=3$, $\Lambda$ is derived equivalent to a hereditary algebra of type $\A_1,\A_2,\A_3$ or $\A_4$;
\item[(iv)] $k=4$ or $5$, $\Lambda$ is a hereditary algebra of type $\A_1,\A_2$;
\item[(v)] $k\geq6$, $\Lambda$ is the a hereditary algebra of type $\A_1$.
\end{itemize}
\end{theorem}
\begin{proof}
``If part'': if $\Lambda$ is derived equivalent to a hereditary algebra $H$, then
\cite[Theorem 2.1]{Ric3} shows that $T_{k-1}(\Lambda)$ is derived equivalent to $T_{k-1}(H)$. Without loss of generality, we assume that $\Lambda$ is a hereditary algebra.

It is easy to see that the algebras $T_{k-1}(\Lambda)$ appearing in (i)-(v) are either hereditary algebras or tilted algebras of Dynkin type. Then $\cd^b(\mod(T_{k-1}(\Lambda)))$ has only fintely many non-isomorphic indecomposable objects up to the shift of complexes. From $\underline{\Gproj} (\Lambda_k)\simeq \cd^b(\mod(T_{k-1}(\Lambda)))/\Psi$ in Proposition \ref{prop: for piecewise hereditary}, we obtain that $\Lambda_k$ is CM-finite in each (i)-(v) similar to the proof of Lemma \ref{lem: RSSIM}.

``Only if part'': $\Lambda_k$ is CM-finite. Corollary \ref{theorem equivalent graded CM classification to the ordinary 2} shows that
$\Lambda_k$ is $\Z/k\Z$-graded CM-finite. Together with $\underline{\Gproj}^{\Z/k\Z} (\Lambda_k)\simeq \cd^b(\mod(T_{k-1}(\Lambda)))/\Sigma^2$ by Corollary \ref{cor: 1}, it is easy to see that
$\cd^b(\mod(T_{k-1}(\Lambda)))$ has only finitely many non-isomorphic indecomposable objects up to the shifts of complexes. Note that $\Lambda$ is naturally a quotient algebra of $T_{k-1}(\Lambda)$, and there is an embedding $\cd^b(\mod(\Lambda))\rightarrow \cd^b(\mod(T_{k-1}(\Lambda)))$. So $\cd^b(\mod(\Lambda))$ has only finitely many non-isomorphic indecomposable objects up to the shifts of complexes, together with Lemma \ref{lemma derived finite category}, $\Lambda$ is derived equivalent to a hereditary algebra $H$ of finite representation type. We also assume that $\Lambda$ is a hereditary algebra.
Then we prove the result case by case.

(i) for $k=1$, it is obvious that $\Lambda_1$ is always CM-finite.

(ii) for $k=2$, $\underline{\Gproj}(\Lambda_k)\simeq \cd^b(\mod(\Lambda))/\Sigma$, so $\Lambda_k$ is CM-finite only if $\Lambda$ is of Dynkin type $\A,\D$, or $\E$.

(iii) for $k=3$, if $\Lambda$ is of type $\A_n$ ($n\geq5$), $\D$ or $\E$, then $T_2(\Lambda)$ is of infinite representation type by \cite[Section 5]{LS1}, which implies that $\cd^b(\mod(T_2(\Lambda)))$ has infinitely many non-isomorphic indecomposable objects up to the shifts of complexes. Similar to the proof of Lemma \ref{lem: RSSIM}, $\Lambda_k$ is CM-infinite in this case.

(iv) for $k=4$ and $5$, if the quiver of $\Lambda$ contains a subquiver of type $\A_3$, from Lemma \ref{lemma derived tubular algebras} (ii), it is easy to see that $\cd^b(\mod(T_{k-1}(\Lambda)))$ has infinitely many non-isomorphic indecomposable objects up to the shifts of complexes. The remaining proof is similar to (iii), hence is omitted here.

(v) for $k\geq6$, if the quiver of $\Lambda$ contains a subquiver of type $\A_2$, then $\cd^b(\mod(T_{k-1}(\Lambda)))$ has infinitely many non-isomorphic indecomposable objects up to the shifts of complexes by Lemma \ref{lemma derived tubular algebras} (iv). The remaining proof is similar to (iii), hence is omitted here.
\end{proof}

\begin{corollary}
For any integer $k>0$, $\Lambda_k=\Lambda\otimes R_k$ is CM-finite if and only if $\Lambda_k$ is $\Z$-graded CM-finite.
\end{corollary}
\begin{proof}
From lemma \ref{theorem equivalent graded CM classification to the ordinary}, we only need prove the ``if part''.

If $\Lambda_k$ is $\Z$-graded CM-finite, then $\cd_{sg}(\mod^\Z(\Lambda_k))$ has only finitely many non-isomorphic indecomposable objects up to the shifts of degree. By Proposition \ref{lemma tilting object } and Proposition \ref{proposition degree k equals shift 2}, we have the triangulated equivalence
$\Xi: \cd^b(\mod(T_{k-1}(\Lambda)))\stackrel{\sim}{\longrightarrow} \cd_{sg}(\mod^\Z(\Lambda_k))$ such that $\Xi\circ \Sigma^2\simeq (k)\circ \Xi$. It follows that $\cd^b(\mod(T_{k-1}(\Lambda)))$ has only finitely many non-isomorphic indecomposable objects up to the shifts of complexes. Then there exists a hereditary algebra $H$ of finite representation type such that $\cd^b(\mod(T_{k-1}(\Lambda)))\simeq \cd^b(\mod (H))$ by Lemma \ref{lemma derived finite category}. It follows from Theorem \ref{proposition CM finite} and its proof that $\Lambda_k$ is CM-finite.
\end{proof}

\section{Examples}

In this section, we describe the Auslander-Reiten quiver of $\Gproj(\Lambda_k)$ for some $\Lambda_k$ of finite Cohen-Macaulay type, mainly under the condition that $\Lambda$ is an indecomposable hereditary algebra, i.e., $\Lambda=KQ$, where $Q$ is a connected finite quiver. In this case, for the algebras $\Lambda_k$ appearing in Theorem \ref{proposition CM finite}, cases (i) and (v) are trivial. For case (ii) $k=2$, the Auslander-Reiten quiver of $\Gproj (\Lambda_2)$ is described in \cite{RZ}. For case (iii) $k=3,4,5$, when $\Lambda$ is of type $\A_2$, its Auslander-Reiten quiver is described in \cite{RS}. So we only need to describe the Auslander-Reiten quiver for $k=3$, and $\Lambda$ is hereditary of type $\A_3$ or $\A_4$.

\subsection{} Let $A=KQ/I$ with the bound quiver $(Q,I)$ as Figure 3 shows, where $I$ is the ideal of $KQ$ generated by $\varepsilon_1^3,\varepsilon_2^3,\varepsilon_3^3$, $\beta\varepsilon_2-\varepsilon_3\beta$ and $\alpha\varepsilon_1-\varepsilon_2\alpha$. Then $A\cong \Lambda_3=\Lambda\otimes_K K[X]/(X^3)$ with for some hereditary algebra $\Lambda$ of type $\A_3$.
\begin{center}\setlength{\unitlength}{0.7mm}
 \begin{picture}(50,13)(0,10)
\put(0,-2){$1$}
\put(18,0){\vector(-1,0){14}}
\put(10,-4){$\alpha$}
\put(20,-2){$2$}

\put(40,-2){$3$}
\put(38,0){\vector(-1,0){14}}

\put(30,-5){$\beta$}

\qbezier(-1,1)(-3,3)(-2,5.5)
\qbezier(-2,5.5)(1,9)(4,5.5)
\qbezier(4,5.5)(5,3)(3,1)
\put(3.1,1.4){\vector(-1,-1){0.3}}

\qbezier(19,1)(17,3)(18,5.5)
\qbezier(18,5.5)(21,9)(24,5.5)
\qbezier(24,5.5)(25,3)(23,1)
\put(23.1,1.4){\vector(-1,-1){0.3}}

\qbezier(39,1)(37,3)(38,5.5)
\qbezier(38,5.5)(41,9)(44,5.5)
\qbezier(44,5.5)(45,3)(43,1)
\put(43.1,1.4){\vector(-1,-1){0.3}}

\put(1,10){$\varepsilon_1$}
\put(21,10){$\varepsilon_2$}
\put(41,10){$\varepsilon_3$}
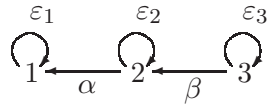
\put(-40,-14){Figure 3. The quiver for $\Lambda_3$ with $\Lambda$ hereditary of type $\A_3$}
\end{picture}
\vspace{1.9cm}
\end{center}
The Auslander-Reiten quiver of $\Lambda_3$ is displayed in Figure 4.
As vertices we have the graded dimension vectors (arising from the obvious $\Z$-covering of $\Lambda_3$) of the indecomposable $\Lambda_3$-modules. The modules on the leftmost column of the figure has to be identified with the corresponding module on the rightmost one. The projective $\Lambda_3$-modules are marked with a solid frame.

\begin{center}\setlength{\unitlength}{0.55mm}
 \begin{picture}(180,150)(0,0)
\put(0,20){\tiny $\setlength{\arraycolsep}{1pt}\renewcommand{\arraystretch}{0.5}{\begin{array}{ccc} 0&0&0\\ 1&0&0\\ 1&1&0\\ 1&1&0  \end{array} }$ }

\put(0,60){\tiny $\setlength{\arraycolsep}{1pt}\renewcommand{\arraystretch}{0.5}{\begin{array}{ccc} 0&0&0\\ 1&0&0\\ 2&1&0\\ 1&1&1  \end{array} }$ }

\put(0,110){\tiny $\setlength{\arraycolsep}{1pt}\renewcommand{\arraystretch}{0.5}{\begin{array}{ccc} 0&0&0\\ 0&0&0\\ 1&1&1\\ 1&1&1  \end{array} }$ }

\put(20,0){\tiny $\setlength{\arraycolsep}{1pt}\renewcommand{\arraystretch}{0.5}{\begin{array}{ccc} 0&0&0\\ 1&1&0\\ 1&1&0\\ 1&1&0  \end{array} }$ }

\put(20,40){\tiny $\setlength{\arraycolsep}{1pt}\renewcommand{\arraystretch}{0.5}{\begin{array}{ccc} 0&0&0\\ 1&0&0\\ 1&1&0\\ 1&1&1  \end{array} }$ }

\put(20,70){\tiny $\setlength{\arraycolsep}{1pt}\renewcommand{\arraystretch}{0.5}{\begin{array}{ccc} 0&0&0\\ 0&0&0\\ 1&1&0\\ 0&0&0  \end{array} }$ }

\put(20,90){\tiny $\setlength{\arraycolsep}{1pt}\renewcommand{\arraystretch}{0.5}{\begin{array}{ccc} 0&0&0\\ 1&0&0\\ 2&1&1\\ 1&1&1  \end{array} }$ }

\put(40,20){\tiny $\setlength{\arraycolsep}{1pt}\renewcommand{\arraystretch}{0.5}{\begin{array}{ccc} 0&0&0\\ 1&1&0\\ 1&1&0\\ 1&1&1  \end{array} }$ }

\put(40,60){\tiny $\setlength{\arraycolsep}{1pt}\renewcommand{\arraystretch}{0.5}{\begin{array}{ccc} 0&0&0\\ 1&0&0\\ 2&2&1\\ 1&1&1  \end{array} }$ }

\put(40,110){\tiny $\setlength{\arraycolsep}{1pt}\renewcommand{\arraystretch}{0.5}{\begin{array}{ccc} 0&0&0\\ 1&0&0\\ 1&0&0\\ 0&0&0  \end{array} }$ }

\put(60,40){\tiny $\setlength{\arraycolsep}{1pt}\renewcommand{\arraystretch}{0.5}{\begin{array}{ccc} 0&0&0\\ 1&1&0\\ 2&2&1\\ 1&1&1  \end{array} }$ }

\put(60,70){\tiny $\setlength{\arraycolsep}{1pt}\renewcommand{\arraystretch}{0.5}{\begin{array}{ccc} 0&0&0\\ 1&0&0\\ 1&1&1\\ 1&1&1  \end{array} }$ }

\put(60,90){\tiny $\setlength{\arraycolsep}{1pt}\renewcommand{\arraystretch}{0.5}{\begin{array}{ccc} 0&0&0\\ 1&0&0\\ 1&1&0\\ 0&0&0  \end{array} }$ }
\put(60,130){\tiny $\setlength{\arraycolsep}{1pt}\renewcommand{\arraystretch}{0.5}{\begin{array}{ccc} 1&0&0\\ 1&0&0\\ 1&0&0\\ 0&0&0  \end{array} }$ }

\put(80,20){\tiny $\setlength{\arraycolsep}{1pt}\renewcommand{\arraystretch}{0.5}{\begin{array}{ccc} 0&0&0\\ 0&0&0\\ 1&1&1\\ 0&0&0  \end{array} }$ }

\put(80,60){\tiny $\setlength{\arraycolsep}{1pt}\renewcommand{\arraystretch}{0.5}{\begin{array}{ccc} 0&0&0\\ 2&1&0\\ 2&2&1\\ 1&1&1  \end{array} }$ }

\put(80,110){\tiny $\setlength{\arraycolsep}{1pt}\renewcommand{\arraystretch}{0.5}{\begin{array}{ccc} 1&0&0\\ 1&0&0\\ 1&1&0\\ 0&0&0  \end{array} }$ }

\put(100,40){\tiny $\setlength{\arraycolsep}{1pt}\renewcommand{\arraystretch}{0.5}{\begin{array}{ccc} 0&0&0\\ 1&0&0\\ 1&1&1\\ 0&0&0  \end{array} }$ }

\put(100,70){\tiny $\setlength{\arraycolsep}{1pt}\renewcommand{\arraystretch}{0.5}{\begin{array}{ccc} 0&0&0\\ 1&1&0\\ 1&1&0\\ 0&0&0  \end{array} }$ }

\put(100,90){\tiny $\setlength{\arraycolsep}{1pt}\renewcommand{\arraystretch}{0.5}{\begin{array}{ccc} 1&0&0\\ 2&1&0\\ 2&2&1\\ 1&1&1  \end{array} }$ }

\put(120,20){\tiny $\setlength{\arraycolsep}{1pt}\renewcommand{\arraystretch}{0.5}{\begin{array}{ccc} 0&0&0\\ 1&0&0\\ 0&0&0\\ 0&0&0  \end{array} }$ }

\put(120,60){\tiny $\setlength{\arraycolsep}{1pt}\renewcommand{\arraystretch}{0.5}{\begin{array}{ccc} 1&0&0\\ 2&1&0\\ 2&2&1\\ 0&0&0  \end{array} }$ }

\put(120,110){\tiny $\setlength{\arraycolsep}{1pt}\renewcommand{\arraystretch}{0.5}{\begin{array}{ccc} 0&0&0\\ 1&1&0\\ 1&1&1\\ 1&1&1  \end{array} }$ }

\put(140,40){\tiny $\setlength{\arraycolsep}{1pt}\renewcommand{\arraystretch}{0.5}{\begin{array}{ccc} 1&0&0\\ 2&1&0\\ 1&1&0\\ 0&0&0  \end{array} }$ }

\put(140,70){\tiny $\setlength{\arraycolsep}{1pt}\renewcommand{\arraystretch}{0.5}{\begin{array}{ccc} 1&0&0\\ 1&0&0\\ 1&1&1\\ 0&0&0  \end{array} }$ }

\put(140,90){\tiny $\setlength{\arraycolsep}{1pt}\renewcommand{\arraystretch}{0.5}{\begin{array}{ccc}  0&0&0\\ 1&1&0\\ 1&1&1\\0&0&0  \end{array} }$ }
\put(140,130){\tiny $\setlength{\arraycolsep}{1pt}\renewcommand{\arraystretch}{0.5}{\begin{array}{ccc} 0&0&0\\ 1&1&1\\ 1&1&1\\ 1&1&1  \end{array} }$ }

\put(160,20){\tiny $\setlength{\arraycolsep}{1pt}\renewcommand{\arraystretch}{0.5}{\begin{array}{ccc}  1&0&0\\ 1&1&0\\ 1&1&0\\ 0&0&0 \end{array} }$ }

\put(160,60){\tiny $\setlength{\arraycolsep}{1pt}\renewcommand{\arraystretch}{0.5}{\begin{array}{ccc}1&0&0\\ 2&1&0\\ 1&1&1 \\ 0&0&0  \end{array} }$ }

\put(160,110){\tiny $\setlength{\arraycolsep}{1pt}\renewcommand{\arraystretch}{0.5}{\begin{array}{ccc}  0&0&0\\ 1&1&1\\ 1&1&1 \\0&0&0 \end{array} }$ }

\put(11,26){\vector(1,1){9}}
\put(11,16){\vector(1,-1){9}}

\put(11,63){\vector(2,1){9}}
\put(11,56){\vector(1,-1){9}}
\put(11,67){\vector(1,2){9}}

\put(11,106){\vector(1,-1){9}}

\put(51,26){\vector(1,1){9}}

\put(51,63){\vector(2,1){9}}
\put(51,56){\vector(1,-1){9}}
\put(51,67){\vector(1,2){9}}

\put(51,106){\vector(1,-1){9}}

\put(51,116){\vector(1,1){9}}

\put(91,26){\vector(1,1){9}}

\put(91,63){\vector(2,1){9}}
\put(91,56){\vector(1,-1){9}}
\put(91,67){\vector(1,2){9}}

\put(91,106){\vector(1,-1){9}}
\put(131,116){\vector(1,1){9}}
\put(131,26){\vector(1,1){9}}

\put(131,63){\vector(2,1){9}}
\put(131,56){\vector(1,-1){9}}
\put(131,67){\vector(1,2){9}}

\put(131,106){\vector(1,-1){9}}

\put(31,6){\vector(1,1){9}}

\put(31,36){\vector(1,-1){9}}

\put(31,46){\vector(1,1){9}}

\put(31,67){\vector(2,-1){9}}

\put(31,85){\vector(1,-2){9}}

\put(31,96){\vector(1,1){9}}

\put(31,6){\vector(1,1){9}}

\put(71,36){\vector(1,-1){9}}

\put(71,46){\vector(1,1){9}}

\put(71,67){\vector(2,-1){9}}

\put(71,85){\vector(1,-2){9}}

\put(71,96){\vector(1,1){9}}
\put(71,126){\vector(1,-1){9}}

\put(111,36){\vector(1,-1){9}}

\put(111,46){\vector(1,1){9}}

\put(111,67){\vector(2,-1){9}}

\put(111,85){\vector(1,-2){9}}

\put(111,96){\vector(1,1){9}}

\put(151,36){\vector(1,-1){9}}

\put(151,46){\vector(1,1){9}}

\put(151,67){\vector(2,-1){9}}

\put(151,85){\vector(1,-2){9}}

\put(151,96){\vector(1,1){9}}
\put(151,126){\vector(1,-1){9}}

\put(20,-7){\begin{picture}(10,20)
\put(0,0){\line(1,0){12}}
\put(0,0){\line(0,1){17}}
\put(12,0){\line(0,1){17}}
\put(0,17){\line(1,0){12}}
\end{picture}}

\put(60,123){\begin{picture}(10,20)
\put(0,0){\line(1,0){12}}
\put(0,0){\line(0,1){17}}
\put(12,0){\line(0,1){17}}
\put(0,17){\line(1,0){12}}
\end{picture}}

\put(140,123){\begin{picture}(10,20)
\put(0,0){\line(1,0){12}}
\put(0,0){\line(0,1){17}}
\put(12,0){\line(0,1){17}}
\put(0,17){\line(1,0){12}}
\end{picture}}

\dashline{3}(5,40)(20,40)
\put(5,40){\vector(-1,0){1}}
\dashline{3}(32,40)(59,40)
\put(32,40){\vector(-1,0){1}}

\dashline{3}(72,40)(99,40)
\put(72,40){\vector(-1,0){1}}

\dashline{3}(112,40)(139,40)
\put(112,40){\vector(-1,0){1}}

\dashline{3}(152,40)(167,40)
\put(152,40){\vector(-1,0){1}}

\dashline{3}(5,90)(20,90)
\put(5,90){\vector(-1,0){1}}
\dashline{3}(32,90)(59,90)
\put(32,90){\vector(-1,0){1}}

\dashline{3}(72,90)(99,90)
\put(72,90){\vector(-1,0){1}}

\dashline{3}(112,90)(139,90)
\put(112,90){\vector(-1,0){1}}

\dashline{3}(152,90)(167,90)
\put(152,90){\vector(-1,0){1}}

\dashline{3}(40,60)(12,60)
\put(12,60){\vector(-1,0){1}}

\dashline{3}(80,60)(52,60)
\put(52,60){\vector(-1,0){1}}

\dashline{3}(120,60)(92,60)
\put(92,60){\vector(-1,0){1}}

\dashline{3}(160,60)(132,60)
\put(132,60){\vector(-1,0){1}}

\dashline{3}(40,20)(12,20)
\put(12,20){\vector(-1,0){1}}

\dashline{3}(80,20)(52,20)
\put(52,20){\vector(-1,0){1}}

\dashline{3}(120,20)(92,20)
\put(92,20){\vector(-1,0){1}}

\dashline{3}(160,20)(132,20)
\put(132,20){\vector(-1,0){1}}

\dashline{3}(40,110)(12,110)
\put(12,110){\vector(-1,0){1}}

\dashline{3}(80,110)(52,110)
\put(52,110){\vector(-1,0){1}}

\dashline{3}(120,110)(92,110)
\put(92,110){\vector(-1,0){1}}

\dashline{3}(160,110)(132,110)
\put(132,110){\vector(-1,0){1}}

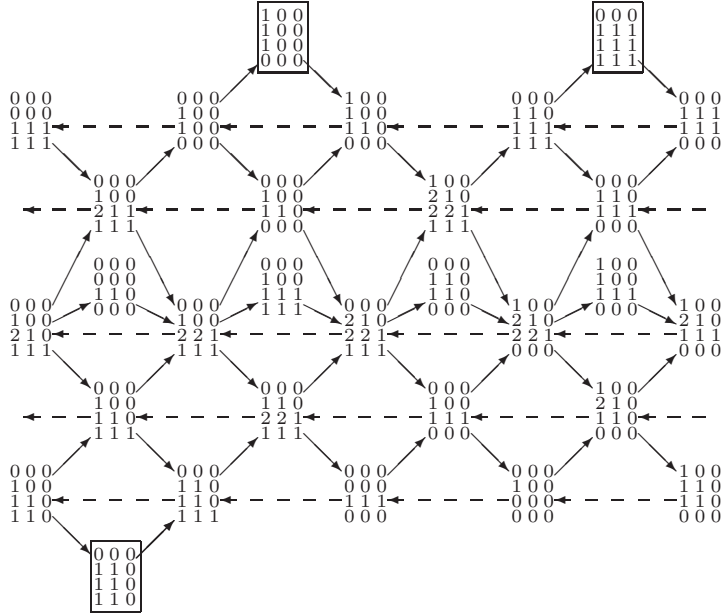
\put(-20,-20){Figure 4. Auslander-Reiten quiver for $\Gproj(\Lambda_3)$ with $\Lambda$ hereditary of type $\A_3$}
\end{picture}

\vspace{1.2cm}
\end{center}

\subsection{} Let $Q$ be the quiver as Figure 5 shows. Let $A=KQ/I$ be the algebra with $I$ the ideal of $KQ$ generated by the relations $\varepsilon_1^3,\varepsilon_2^3,\varepsilon_3^3,\varepsilon_4^3$, $\alpha\varepsilon_1- \varepsilon_2\alpha$, $\beta\varepsilon_2-\varepsilon_3\beta$  and $\gamma\varepsilon_3-\varepsilon_4\gamma$. Then $A\cong\Lambda_3=\Lambda\otimes_K K[X]/(X^3)$ for some hereditary algebra $\Lambda$ of type $\A_4$. The Auslander-Reiten quiver of $\Lambda_3$ is displayed in Figure 6.
The modules on the leftmost column of the figure has to be identified with the corresponding module on the rightmost one.

\begin{center}\setlength{\unitlength}{0.7mm}
 \begin{picture}(50,10)(0,10)
\put(0,-2){$1$}
\put(18,0){\vector(-1,0){14}}
\put(10,-4){$\alpha$}
\put(20,-2){$2$}

\put(40,-2){$3$}
\put(38,0){\vector(-1,0){14}}

\put(30,-5){$\beta$}

\put(50,-5){$\gamma$}

\put(60,-2){$4$}
\put(58,0){\vector(-1,0){14}}

\qbezier(59,1)(57,3)(58,5.5)
\qbezier(58,5.5)(61,9)(64,5.5)
\qbezier(64,5.5)(65,3)(63,1)
\put(63.1,1.4){\vector(-1,-1){0.3}}

\qbezier(-1,1)(-3,3)(-2,5.5)
\qbezier(-2,5.5)(1,9)(4,5.5)
\qbezier(4,5.5)(5,3)(3,1)
\put(3.1,1.4){\vector(-1,-1){0.3}}

\qbezier(19,1)(17,3)(18,5.5)
\qbezier(18,5.5)(21,9)(24,5.5)
\qbezier(24,5.5)(25,3)(23,1)
\put(23.1,1.4){\vector(-1,-1){0.3}}

\qbezier(39,1)(37,3)(38,5.5)
\qbezier(38,5.5)(41,9)(44,5.5)
\qbezier(44,5.5)(45,3)(43,1)
\put(43.1,1.4){\vector(-1,-1){0.3}}

\put(1,10){$\varepsilon_1$}
\put(21,10){$\varepsilon_2$}
\put(41,10){$\varepsilon_3$}
\put(61,10){$\varepsilon_4$}
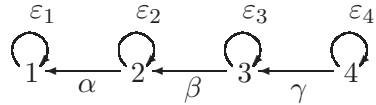
\put(-40,-14){Figure 5. The quiver for $\Lambda_3$ with $\Lambda$ hereditary of type $\A_4$}

\end{picture}
\vspace{1.6cm}
\end{center}

\begin{landscape}

\begin{center}\setlength{\unitlength}{0.56mm}
 \begin{picture}(440,240)(20,-20)
\put(0,20){\tiny $\setlength{\arraycolsep}{1pt}\renewcommand{\arraystretch}{0.5}{\begin{array}{cccc} 0&0&0&0\\ 1&0&0&0\\ 1&0&0&0\\ 0&0&0&0  \end{array} }$ }

\put(40,20){\tiny $\setlength{\arraycolsep}{1pt}\renewcommand{\arraystretch}{0.5}{\begin{array}{cccc} 1&0&0&0\\ 1&0&0&0\\ 1&1&0&0\\ 0&0&0&0  \end{array} }$ }

\put(80,20){\tiny $\setlength{\arraycolsep}{1pt}\renewcommand{\arraystretch}{0.5}{\begin{array}{cccc} 0&0&0&0\\ 1&1&0&0\\ 1&1&1&0\\ 1&1&1&0  \end{array} }$ }

\put(120,20){\tiny $\setlength{\arraycolsep}{1pt}\renewcommand{\arraystretch}{0.5}{\begin{array}{cccc} 1&0&0&0\\ 1&1&1&0\\ 1&1&1&0\\ 1&1&1&1  \end{array} }$ }

\put(160,20){\tiny $\setlength{\arraycolsep}{1pt}\renewcommand{\arraystretch}{0.5}{\begin{array}{cccc} 0&0&0&0\\ 0&0&0&0\\ 1&1&1&1\\ 0&0&0&0  \end{array} }$ }

\put(200,20){\tiny $\setlength{\arraycolsep}{1pt}\renewcommand{\arraystretch}{0.5}{\begin{array}{cccc} 0&0&0&0\\ 1&0&0&0\\ 0&0&0&0\\ 0&0&0&0  \end{array} }$ }

\put(240,20){\tiny $\setlength{\arraycolsep}{1pt}\renewcommand{\arraystretch}{0.5}{\begin{array}{cccc} 1&0&0&0\\ 1&1&0&0\\ 1&1&0&0\\ 0&0&0&0  \end{array} }$ }

\put(280,20){\tiny $\setlength{\arraycolsep}{1pt}\renewcommand{\arraystretch}{0.5}{\begin{array}{cccc} 1&1&0&0\\ 1&1&0&0\\ 1&1&1&0\\ 0&0&0&0  \end{array} }$ }

\put(320,20){\tiny $\setlength{\arraycolsep}{1pt}\renewcommand{\arraystretch}{0.5}{\begin{array}{cccc} 0&0&0&0\\ 1&1&1&0\\ 1&1&1&1\\ 1&1&1&1  \end{array} }$ }

\put(360,20){\tiny $\setlength{\arraycolsep}{1pt}\renewcommand{\arraystretch}{0.5}{\begin{array}{cccc}0&0&0&0\\ 1&1&1&1\\ 1&1&1&1\\ 0&0&0&0  \end{array} }$ }

\put(400,20){\tiny $\setlength{\arraycolsep}{1pt}\renewcommand{\arraystretch}{0.5}{\begin{array}{cccc} 1&0&0&0\\ 1&0&0&0\\ 0&0&0&0\\ 0&0&0&0  \end{array} }$ }

\put(0,60){\tiny $\setlength{\arraycolsep}{1pt}\renewcommand{\arraystretch}{0.5}{\begin{array}{cccc} 0&0&0&0\\ 1&0&0&0\\ 2&2&1&1\\ 1&1&1&1  \end{array} }$ }

\put(40,60){\tiny $\setlength{\arraycolsep}{1pt}\renewcommand{\arraystretch}{0.5}{\begin{array}{cccc} 0&0&0&0\\ 2&1&0&0\\ 2&2&1&0\\ 1&1&1&0  \end{array} }$ }

\put(80,60){\tiny $\setlength{\arraycolsep}{1pt}\renewcommand{\arraystretch}{0.5}{\begin{array}{cccc} 1&0&0&0\\ 2&1&0&0\\ 2&2&1&0\\ 1&1&1&1  \end{array} }$ }

\put(120,60){\tiny $\setlength{\arraycolsep}{1pt}\renewcommand{\arraystretch}{0.5}{\begin{array}{cccc} 0&0&0&0\\ 1&1&0&0\\ 2&2&2&1\\ 1&1&1&1  \end{array} }$ }

\put(160,60){\tiny $\setlength{\arraycolsep}{1pt}\renewcommand{\arraystretch}{0.5}{\begin{array}{cccc} 0&0&0&0\\ 2&1&1&0\\ 2&2&2&1\\ 1&1&1&1  \end{array} }$ }

\put(200,60){\tiny $\setlength{\arraycolsep}{1pt}\renewcommand{\arraystretch}{0.5}{\begin{array}{cccc} 1&0&0&0\\ 2&1&0&0\\ 2&2&1&1\\ 0&0&0&0  \end{array} }$ }

\put(240,60){\tiny $\setlength{\arraycolsep}{1pt}\renewcommand{\arraystretch}{0.5}{\begin{array}{cccc} 1&0&0&0\\ 2&1&0&0\\ 1&1&1&0\\ 0&0&0&0  \end{array} }$ }

\put(280,60){\tiny $\setlength{\arraycolsep}{1pt}\renewcommand{\arraystretch}{0.5}{\begin{array}{cccc} 1&0&0&0\\ 2&2&  1&0\\ 2&2&2&1\\ 1&1&1&1  \end{array} }$ }

\put(320,60){\tiny $\setlength{\arraycolsep}{1pt}\renewcommand{\arraystretch}{0.5}{\begin{array}{cccc} 1&1&0&0\\ 2&2&1&0\\ 2&2&2&1\\ 0&0&0&0  \end{array} }$ }

\put(360,60){\tiny $\setlength{\arraycolsep}{1pt}\renewcommand{\arraystretch}{0.5}{\begin{array}{cccc}1&0&0&0\\ 2&1&1&0\\ 1&1&1&1\\ 0&0&0&0  \end{array} }$ }

\put(400,60){\tiny $\setlength{\arraycolsep}{1pt}\renewcommand{\arraystretch}{0.5}{\begin{array}{cccc} 1&0&0&0\\ 2&2&1&1\\ 1&1&1&1\\ 0&0&0&0  \end{array} }$ }

\put(0,100){\tiny $\setlength{\arraycolsep}{1pt}\renewcommand{\arraystretch}{0.5}{\begin{array}{cccc} 0&0&0&0\\ 2&1&0&0\\ 3&3&2&0\\ 2&2&2&1  \end{array} }$ }

\put(40,100){\tiny $\setlength{\arraycolsep}{1pt}\renewcommand{\arraystretch}{0.5}{\begin{array}{cccc} 0&0&0&0\\ 2&1&0&0\\ 3&3&2&1\\ 2&2&2&2  \end{array} }$ }

\put(80,100){\tiny $\setlength{\arraycolsep}{1pt}\renewcommand{\arraystretch}{0.5}{\begin{array}{cccc} 0&0&0&0\\ 2&1&0&0\\ 3&3&2&1\\ 1&1&1&1  \end{array} }$ }

\put(120,100){\tiny $\setlength{\arraycolsep}{1pt}\renewcommand{\arraystretch}{0.5}{\begin{array}{cccc} 1&0&0&0\\ 3&1&0&0\\ 3&3&2&1\\ 1&1&1&1  \end{array} }$ }

\put(160,100){\tiny $\setlength{\arraycolsep}{1pt}\renewcommand{\arraystretch}{0.5}{\begin{array}{cccc} 1&0&0&0\\ 3&2&0&0\\ 3&3&2&1\\ 1&1&1&1  \end{array} }$ }

\put(200,100){\tiny $\setlength{\arraycolsep}{1pt}\renewcommand{\arraystretch}{0.5}{\begin{array}{cccc} 1&0&0&0\\ 3&2&1&0\\ 3&3&3&1\\ 1&1&1&1  \end{array} }$ }

\put(240,100){\tiny $\setlength{\arraycolsep}{1pt}\renewcommand{\arraystretch}{0.5}{\begin{array}{cccc} 1&0&0&0\\ 3&2&1&0\\ 3&3&3&2\\ 1&1&1&1  \end{array} }$ }

\put(280,100){\tiny $\setlength{\arraycolsep}{1pt}\renewcommand{\arraystretch}{0.5}{\begin{array}{cccc} 1&0&0&0\\ 3&2&1&0\\ 2&2&2&1\\ 0&0&0&0  \end{array} }$ }

\put(320,100){\tiny $\setlength{\arraycolsep}{1pt}\renewcommand{\arraystretch}{0.5}{\begin{array}{cccc} 2&0&0&0\\ 3&2&1&0\\ 2&2&2&1\\ 0&0&0&0  \end{array} }$ }

\put(360,100){\tiny $\setlength{\arraycolsep}{1pt}\renewcommand{\arraystretch}{0.5}{\begin{array}{cccc}2&1&0&0\\ 3&3&1&0\\ 2&2&2&1\\ 0&0&0&0  \end{array} }$ }

\put(400,100){\tiny $\setlength{\arraycolsep}{1pt}\renewcommand{\arraystretch}{0.5}{\begin{array}{cccc} 2&1&0&0\\ 3&3&2&0\\ 2&2&2&1\\ 0&0&0&0  \end{array} }$ }

\put(0,150){\tiny $\setlength{\arraycolsep}{1pt}\renewcommand{\arraystretch}{0.5}{\begin{array}{cccc} 0&0&0&0\\ 1&1&0&0\\ 1&1&0&0\\ 1&1&1&1  \end{array} }$ }

\put(40,150){\tiny $\setlength{\arraycolsep}{1pt}\renewcommand{\arraystretch}{0.5}{\begin{array}{cccc} 0&0&0&0\\ 0&0&0&0\\ 1&1&1&0\\ 0&0&0&0  \end{array} }$ }

\put(80,150){\tiny $\setlength{\arraycolsep}{1pt}\renewcommand{\arraystretch}{0.5}{\begin{array}{cccc} 0&0&0&0\\ 1&0&0&0\\ 1&1&1&1\\ 1&1&1&1  \end{array} }$ }

\put(120,150){\tiny $\setlength{\arraycolsep}{1pt}\renewcommand{\arraystretch}{0.5}{\begin{array}{cccc} 0&0&0&0\\ 1&1&0&0\\ 1&1&0&0\\ 0&0&0&0  \end{array} }$ }

\put(160,150){\tiny $\setlength{\arraycolsep}{1pt}\renewcommand{\arraystretch}{0.5}{\begin{array}{cccc} 1&0&0&0\\ 1&0&0&0\\ 1&1&1&0\\ 0&0&0&0  \end{array} }$ }

\put(200,150){\tiny $\setlength{\arraycolsep}{1pt}\renewcommand{\arraystretch}{0.5}{\begin{array}{cccc} 0&0&0&0\\ 1&1&0&0\\ 1&1&1&1\\ 1&1&1&1 \end{array} }$ }

\put(240,150){\tiny $\setlength{\arraycolsep}{1pt}\renewcommand{\arraystretch}{0.5}{\begin{array}{cccc} 0&0&0&0\\ 1&1&1&0\\ 1&1&1&0\\ 0&0&0&0  \end{array} }$ }

\put(280,150){\tiny $\setlength{\arraycolsep}{1pt}\renewcommand{\arraystretch}{0.5}{\begin{array}{cccc} 1&0&0&0\\ 1&0&0&0\\ 1&1&1&1\\ 0&0&0&0  \end{array} }$ }

\put(320,150){\tiny $\setlength{\arraycolsep}{1pt}\renewcommand{\arraystretch}{0.5}{\begin{array}{cccc} 0&0&0&0\\ 1&1&0&0\\ 0&0&0&0\\ 0&0&0&0  \end{array} }$ }

\put(360,150){\tiny $\setlength{\arraycolsep}{1pt}\renewcommand{\arraystretch}{0.5}{\begin{array}{cccc}1&0&0&0\\ 1&1&1&0\\ 1&1&1&0\\ 0&0&0&0  \end{array} }$ }

\put(400,150){\tiny $\setlength{\arraycolsep}{1pt}\renewcommand{\arraystretch}{0.5}{\begin{array}{cccc} 1&1&0&0\\ 1&1&0&0\\ 1&1&1&1\\ 0&0&0&0  \end{array} }$ }

\put(20,40){\tiny $\setlength{\arraycolsep}{1pt}\renewcommand{\arraystretch}{0.5}{\begin{array}{cccc} 0&0&0&0\\ 1&0&0&0\\ 1&1&0&0\\ 0&0&0&0  \end{array} }$ }

\put(60,40){\tiny $\setlength{\arraycolsep}{1pt}\renewcommand{\arraystretch}{0.5}{\begin{array}{cccc} 1&0&0&0\\ 2&1&0&0\\ 2&2&1&0\\ 1&1&1&0  \end{array} }$ }

\put(100,40){\tiny $\setlength{\arraycolsep}{1pt}\renewcommand{\arraystretch}{0.5}{\begin{array}{cccc} 0&0&0&0\\ 1&1&0&0\\ 1&1&1&0\\ 1&1&1&1  \end{array} }$ }
\put(140,40){\tiny $\setlength{\arraycolsep}{1pt}\renewcommand{\arraystretch}{0.5}{\begin{array}{cccc} 0&0&0&0\\ 1&1&1&0\\ 2&2&2&1\\ 1&1&1&1  \end{array} }$ }

\put(180,40){\tiny $\setlength{\arraycolsep}{1pt}\renewcommand{\arraystretch}{0.5}{\begin{array}{cccc} 0&0&0&0\\ 1&0&0&0\\ 1&1&1&1\\ 0&0&0&0  \end{array} }$ }

\put(220,40){\tiny $\setlength{\arraycolsep}{1pt}\renewcommand{\arraystretch}{0.5}{\begin{array}{cccc} 1&0&0&0\\ 2&1&0&0\\ 1&1&0&0\\ 0&0&0&0  \end{array} }$ }

\put(260,40){\tiny $\setlength{\arraycolsep}{1pt}\renewcommand{\arraystretch}{0.5}{\begin{array}{cccc} 1&0&0&0\\ 1&1&0&0\\ 1&1&1&0\\ 0&0&0&0  \end{array} }$ }

\put(300,40){\tiny $\setlength{\arraycolsep}{1pt}\renewcommand{\arraystretch}{0.5}{\begin{array}{cccc} 1&1&0&0\\ 2&2&1&0\\ 2&2&2&1\\ 1&1&1&1  \end{array} }$ }

\put(340,40){\tiny $\setlength{\arraycolsep}{1pt}\renewcommand{\arraystretch}{0.5}{\begin{array}{cccc} 0&0&0&0\\ 1&1&1&0\\ 1&1&1&1\\ 0&0&0&0  \end{array} }$ }

\put(380,40){\tiny $\setlength{\arraycolsep}{1pt}\renewcommand{\arraystretch}{0.5}{\begin{array}{cccc} 1&0&0&0\\ 2&1&1&1\\ 1&1&1&1\\ 0&0&0&0  \end{array} }$ }

\put(20,80){\tiny $\setlength{\arraycolsep}{1pt}\renewcommand{\arraystretch}{0.5}{\begin{array}{cccc} 0&0&0&0\\ 2&1&0&0\\ 3&3&2&1\\ 2&2&2&1  \end{array} }$ }

\put(60,80){\tiny $\setlength{\arraycolsep}{1pt}\renewcommand{\arraystretch}{0.5}{\begin{array}{cccc} 0&0&0&0\\ 2&1&0&0\\ 2&2&1&0\\ 1&1&1&1  \end{array} }$ }

\put(100,80){\tiny $\setlength{\arraycolsep}{1pt}\renewcommand{\arraystretch}{0.5}{\begin{array}{cccc} 1&0&0&0\\ 2&1&0&0\\ 3&3&2&1\\ 1&1&1&1  \end{array} }$ }
\put(140,80){\tiny $\setlength{\arraycolsep}{1pt}\renewcommand{\arraystretch}{0.5}{\begin{array}{cccc} 0&0&0&0\\ 2&1&0&0\\ 2&2&2&1\\ 1&1&1&1  \end{array} }$ }

\put(180,80){\tiny $\setlength{\arraycolsep}{1pt}\renewcommand{\arraystretch}{0.5}{\begin{array}{cccc} 1&0&0&0\\ 3&2&1&0\\ 3&3&2&1\\  1&1&1&1   \end{array} }$ }

\put(220,80){\tiny $\setlength{\arraycolsep}{1pt}\renewcommand{\arraystretch}{0.5}{\begin{array}{cccc} 1&0&0&0\\ 2&1&0&0\\ 2&2&2&1\\ 0&0&0&0  \end{array} }$ }

\put(260,80){\tiny $\setlength{\arraycolsep}{1pt}\renewcommand{\arraystretch}{0.5}{\begin{array}{cccc} 1&0&0&0\\ 3&2&1&0\\ 2&2&2&1\\ 1&1&1&1  \end{array} }$ }

\put(300,80){\tiny $\setlength{\arraycolsep}{1pt}\renewcommand{\arraystretch}{0.5}{\begin{array}{cccc} 1&0&0&0\\ 2&2&1&0\\ 2&2&2&1\\ 0&0&0&0  \end{array} }$ }

\put(340,80){\tiny $\setlength{\arraycolsep}{1pt}\renewcommand{\arraystretch}{0.5}{\begin{array}{cccc} 2&1&0&0\\ 3&2&1&0\\ 2&2&2&1\\ 0&0&0&0  \end{array} }$ }

\put(380,80){\tiny $\setlength{\arraycolsep}{1pt}\renewcommand{\arraystretch}{0.5}{\begin{array}{cccc} 1&0&0&0\\ 2&2&1&0\\ 1&1&1&1\\ 0&0&0&0  \end{array} }$ }

\put(20,110){\tiny $\setlength{\arraycolsep}{1pt}\renewcommand{\arraystretch}{0.5}{\begin{array}{cccc} 0&0&0&0\\ 1&0&0&0\\ 1&1&1&0\\ 1&1&1&1  \end{array} }$ }

\put(60,110){\tiny $\setlength{\arraycolsep}{1pt}\renewcommand{\arraystretch}{0.5}{\begin{array}{cccc} 0&0&0&0\\ 1&1&0&0\\ 2&2&1&1\\ 1&1&1&1  \end{array} }$ }

\put(100,110){\tiny $\setlength{\arraycolsep}{1pt}\renewcommand{\arraystretch}{0.5}{\begin{array}{cccc} 0&0&0&0\\ 1&0&0&0\\ 1&1&1&0\\ 0&0&0&0  \end{array} }$ }

\put(140,110){\tiny $\setlength{\arraycolsep}{1pt}\renewcommand{\arraystretch}{0.5}{\begin{array}{cccc} 1&0&0&0\\ 2&1&0&0\\ 2&2&1&1\\  1&1&1&1   \end{array} }$ }

\put(180,110){\tiny $\setlength{\arraycolsep}{1pt}\renewcommand{\arraystretch}{0.5}{\begin{array}{cccc} 0&0&0&0\\ 1&1&0&0\\ 1&1&1&0\\ 0&0&0&0  \end{array} }$ }

\put(220,110){\tiny $\setlength{\arraycolsep}{1pt}\renewcommand{\arraystretch}{0.5}{\begin{array}{cccc} 1&0&0&0\\ 2&1&1&0\\ 2&2&2&1\\  1&1&1&1   \end{array} }$ }

\put(260,110){\tiny $\setlength{\arraycolsep}{1pt}\renewcommand{\arraystretch}{0.5}{\begin{array}{cccc} 0&0&0&0\\ 1&1&0&0\\ 1&1&1&1\\ 0&0&0&0  \end{array} }$ }

\put(300,110){\tiny $\setlength{\arraycolsep}{1pt}\renewcommand{\arraystretch}{0.5}{\begin{array}{cccc} 1&0&0&0\\ 2&1&1&0\\ 1&1&1&0\\ 0&0&0&0  \end{array} }$ }

\put(340,110){\tiny $\setlength{\arraycolsep}{1pt}\renewcommand{\arraystretch}{0.5}{\begin{array}{cccc} 1&0&0&0\\ 1&1&0&0\\ 1&1&1&1\\ 0&0&0&0  \end{array} }$ }

\put(380,110){\tiny $\setlength{\arraycolsep}{1pt}\renewcommand{\arraystretch}{0.5}{\begin{array}{cccc} 1&1&0&0\\ 2&2&1&0\\ 1&1&1&0\\ 0&0&0&0  \end{array} }$ }

\put(20,130){\tiny $\setlength{\arraycolsep}{1pt}\renewcommand{\arraystretch}{0.5}{\begin{array}{cccc} 0&0&0&0\\ 1&1&0&0\\ 2&2&1&0\\  1&1&1&1   \end{array} }$ }

\put(60,130){\tiny $\setlength{\arraycolsep}{1pt}\renewcommand{\arraystretch}{0.5}{\begin{array}{cccc} 0&0&0&0\\ 1&0&0&0\\ 2&2&1&0\\  1&1&1&1   \end{array} }$ }

\put(100,130){\tiny $\setlength{\arraycolsep}{1pt}\renewcommand{\arraystretch}{0.5}{\begin{array}{cccc} 0&0&0&0\\ 2&1&0&0\\ 2&2&1&1\\  1&1&1&1   \end{array} }$ }

\put(140,130){\tiny $\setlength{\arraycolsep}{1pt}\renewcommand{\arraystretch}{0.5}{\begin{array}{cccc} 1&0&0&0\\ 2&1&0&0\\ 2&2&1&0\\ 0&0&0&0  \end{array} }$ }

\put(180,130){\tiny $\setlength{\arraycolsep}{1pt}\renewcommand{\arraystretch}{0.5}{\begin{array}{cccc} 1&0&0&0\\ 2&1&0&0\\ 2&2&2&1\\  1&1&1&1  \end{array} }$ }

\put(220,130){\tiny $\setlength{\arraycolsep}{1pt}\renewcommand{\arraystretch}{0.5}{\begin{array}{cccc} 0&0&0&0\\ 2&2&1&0\\ 2&2&2&1\\  1&1&1&1   \end{array} }$ }

\put(260,130){\tiny $\setlength{\arraycolsep}{1pt}\renewcommand{\arraystretch}{0.5}{\begin{array}{cccc} 1&0&0&0\\ 2&1&1&0\\ 2&2&2&1\\ 0&0&0&0  \end{array} }$ }

\put(300,130){\tiny $\setlength{\arraycolsep}{1pt}\renewcommand{\arraystretch}{0.5}{\begin{array}{cccc} 1&0&0&0\\ 2&1&0&0\\ 1&1&1&1\\ 0&0&0&0  \end{array} }$ }

\put(340,130){\tiny $\setlength{\arraycolsep}{1pt}\renewcommand{\arraystretch}{0.5}{\begin{array}{cccc} 1&0&0&0\\ 2&2&1&0\\ 1&1&1&0\\ 0&0&0&0  \end{array} }$ }

\put(380,130){\tiny $\setlength{\arraycolsep}{1pt}\renewcommand{\arraystretch}{0.5}{\begin{array}{cccc} 2&1&0&0\\ 2&2&1&0\\ 2&2&2&1\\ 0&0&0&0  \end{array} }$ }

\put(20,0){\tiny $\setlength{\arraycolsep}{1pt}\renewcommand{\arraystretch}{0.5}{\begin{array}{cccc} 1&0&0&0\\ 1&0&0&0\\ 1&0&0&0\\ 0&0&0&0  \end{array} }$ }

\put(100,0){\tiny $\setlength{\arraycolsep}{1pt}\renewcommand{\arraystretch}{0.5}{\begin{array}{cccc} 0&0&0&0 \\1&1&1&0\\ 1&1&1&0\\ 1&1&1&0  \end{array} }$ }

\put(260,0){\tiny $\setlength{\arraycolsep}{1pt}\renewcommand{\arraystretch}{0.5}{\begin{array}{cccc} 1&1&0&0 \\1&1&0&0\\ 1&1&0&0\\0&0&0&0  \end{array} }$ }

\put(340,0){\tiny $\setlength{\arraycolsep}{1pt}\renewcommand{\arraystretch}{0.5}{\begin{array}{cccc} 0&0&0&0 \\1&1&1&1\\ 1&1&1&1\\1&1&1&1  \end{array} }$ }

\put(14,66){\vector(1,1){7}}
\put(14,56){\vector(1,-1){7}}

\put(14,103){\vector(2,1){7}}
\put(14,96){\vector(1,-1){7}}
\put(14,107){\vector(1,3){7}}

\put(14,146){\vector(1,-1){7}}

\put(14,26){\vector(1,1){7}}

\put(54,66){\vector(1,1){7}}
\put(54,56){\vector(1,-1){7}}

\put(54,103){\vector(2,1){7}}
\put(54,96){\vector(1,-1){7}}
\put(54,107){\vector(1,3){7}}

\put(54,146){\vector(1,-1){7}}

\put(54,26){\vector(1,1){7}}

\put(94,66){\vector(1,1){7}}
\put(94,56){\vector(1,-1){7}}

\put(94,103){\vector(2,1){7}}
\put(94,96){\vector(1,-1){7}}
\put(94,107){\vector(1,3){7}}

\put(94,146){\vector(1,-1){7}}

\put(94,26){\vector(1,1){7}}

\put(134,66){\vector(1,1){7}}
\put(134,56){\vector(1,-1){7}}

\put(134,103){\vector(2,1){7}}
\put(134,96){\vector(1,-1){7}}
\put(134,107){\vector(1,3){7}}

\put(134,146){\vector(1,-1){7}}

\put(134,26){\vector(1,1){7}}

\put(174,66){\vector(1,1){7}}
\put(174,56){\vector(1,-1){7}}

\put(174,103){\vector(2,1){7}}
\put(174,96){\vector(1,-1){7}}
\put(174,107){\vector(1,3){7}}

\put(174,146){\vector(1,-1){7}}

\put(174,26){\vector(1,1){7}}

\put(214,66){\vector(1,1){7}}
\put(214,56){\vector(1,-1){7}}

\put(214,103){\vector(2,1){7}}
\put(214,96){\vector(1,-1){7}}
\put(214,107){\vector(1,3){7}}

\put(214,146){\vector(1,-1){7}}

\put(214,26){\vector(1,1){7}}

\put(254,66){\vector(1,1){7}}
\put(254,56){\vector(1,-1){7}}

\put(254,103){\vector(2,1){7}}
\put(254,96){\vector(1,-1){7}}
\put(254,107){\vector(1,3){7}}

\put(254,146){\vector(1,-1){7}}

\put(254,26){\vector(1,1){7}}

\put(294,66){\vector(1,1){7}}
\put(294,56){\vector(1,-1){7}}

\put(294,103){\vector(2,1){7}}
\put(294,96){\vector(1,-1){7}}
\put(294,107){\vector(1,3){7}}

\put(294,146){\vector(1,-1){7}}

\put(294,26){\vector(1,1){7}}

\put(334,66){\vector(1,1){7}}
\put(334,56){\vector(1,-1){7}}

\put(334,103){\vector(2,1){7}}
\put(334,96){\vector(1,-1){7}}
\put(334,107){\vector(1,3){7}}

\put(334,146){\vector(1,-1){7}}

\put(334,26){\vector(1,1){7}}

\put(374,66){\vector(1,1){7}}
\put(374,56){\vector(1,-1){7}}

\put(374,103){\vector(2,1){7}}
\put(374,96){\vector(1,-1){7}}
\put(374,107){\vector(1,3){7}}

\put(374,146){\vector(1,-1){7}}

\put(374,26){\vector(1,1){7}}

\put(34,76){\vector(1,-1){7}}

\put(34,88){\vector(1,1){7}}

\put(34,107){\vector(2,-1){7}}

\put(33,127){\vector(1,-3){7}}

\put(34,138){\vector(1,1){7}}

\put(34,48){\vector(1,1){7}}
\put(34,36){\vector(1,-1){7}}

\put(74,76){\vector(1,-1){7}}

\put(74,88){\vector(1,1){7}}

\put(74,107){\vector(2,-1){7}}

\put(73,127){\vector(1,-3){7}}

\put(74,138){\vector(1,1){7}}

\put(74,48){\vector(1,1){7}}
\put(74,36){\vector(1,-1){7}}

\put(114,76){\vector(1,-1){7}}

\put(114,88){\vector(1,1){7}}

\put(114,107){\vector(2,-1){7}}

\put(113,127){\vector(1,-3){7}}

\put(114,138){\vector(1,1){7}}

\put(114,48){\vector(1,1){7}}
\put(114,36){\vector(1,-1){7}}

\put(154,76){\vector(1,-1){7}}

\put(154,88){\vector(1,1){7}}

\put(154,107){\vector(2,-1){7}}

\put(153,127){\vector(1,-3){7}}

\put(154,138){\vector(1,1){7}}

\put(154,48){\vector(1,1){7}}
\put(154,36){\vector(1,-1){7}}

\put(194,76){\vector(1,-1){7}}

\put(194,88){\vector(1,1){7}}

\put(194,107){\vector(2,-1){7}}

\put(193,127){\vector(1,-3){7}}

\put(194,138){\vector(1,1){7}}

\put(194,48){\vector(1,1){7}}
\put(194,36){\vector(1,-1){7}}

\put(234,76){\vector(1,-1){7}}

\put(234,88){\vector(1,1){7}}

\put(234,107){\vector(2,-1){7}}

\put(233,127){\vector(1,-3){7}}

\put(234,138){\vector(1,1){7}}

\put(234,48){\vector(1,1){7}}
\put(234,36){\vector(1,-1){7}}

\put(274,76){\vector(1,-1){7}}

\put(274,88){\vector(1,1){7}}

\put(274,107){\vector(2,-1){7}}

\put(273,127){\vector(1,-3){7}}

\put(274,138){\vector(1,1){7}}

\put(274,48){\vector(1,1){7}}
\put(274,36){\vector(1,-1){7}}

\put(314,76){\vector(1,-1){7}}

\put(314,88){\vector(1,1){7}}

\put(314,107){\vector(2,-1){7}}

\put(313,127){\vector(1,-3){7}}

\put(314,138){\vector(1,1){7}}

\put(314,48){\vector(1,1){7}}
\put(314,36){\vector(1,-1){7}}

\put(354,76){\vector(1,-1){7}}

\put(354,88){\vector(1,1){7}}

\put(354,107){\vector(2,-1){7}}

\put(353,127){\vector(1,-3){7}}

\put(354,138){\vector(1,1){7}}

\put(354,48){\vector(1,1){7}}
\put(354,36){\vector(1,-1){7}}

\put(394,76){\vector(1,-1){7}}

\put(394,88){\vector(1,1){7}}

\put(394,107){\vector(2,-1){7}}

\put(393,127){\vector(1,-3){7}}

\put(394,138){\vector(1,1){7}}

\put(394,48){\vector(1,1){7}}
\put(394,36){\vector(1,-1){7}}

\put(34,7){\vector(1,1){7}}
\put(14,15){\vector(1,-1){7}}

\put(114,7){\vector(1,1){7}}
\put(94,15){\vector(1,-1){7}}

\put(274,7){\vector(1,1){7}}
\put(254,15){\vector(1,-1){7}}

\put(354,7){\vector(1,1){7}}
\put(334,15){\vector(1,-1){7}}

\put(20,-7){\begin{picture}(10,20)
\put(0,0){\line(1,0){15}}
\put(0,0){\line(0,1){17}}
\put(15,0){\line(0,1){17}}
\put(0,17){\line(1,0){15}}
\end{picture}}

\put(100,-7){\begin{picture}(10,20)
\put(0,0){\line(1,0){15}}
\put(0,0){\line(0,1){17}}
\put(15,0){\line(0,1){17}}
\put(0,17){\line(1,0){15}}
\end{picture}}

\put(260,-7){\begin{picture}(10,20)
\put(0,0){\line(1,0){15}}
\put(0,0){\line(0,1){17}}
\put(15,0){\line(0,1){17}}
\put(0,17){\line(1,0){15}}
\end{picture}}

\put(340,-7){\begin{picture}(10,20)
\put(0,0){\line(1,0){15}}
\put(0,0){\line(0,1){17}}
\put(15,0){\line(0,1){17}}
\put(0,17){\line(1,0){15}}
\end{picture}}

\dashline{3}(76,80)(99,80)
\put(76,80){\vector(-1,0){1}}

\dashline{3}(36,80)(59,80)
\put(36,80){\vector(-1,0){1}}

\dashline{3}(116,80)(139,80)
\put(116,80){\vector(-1,0){1}}

\dashline{3}(156,80)(179,80)
\put(156,80){\vector(-1,0){1}}

\dashline{3}(196,80)(219,80)
\put(196,80){\vector(-1,0){1}}

\dashline{3}(236,80)(259,80)
\put(236,80){\vector(-1,0){1}}

\dashline{3}(276,80)(299,80)
\put(276,80){\vector(-1,0){1}}

\dashline{3}(316,80)(339,80)
\put(316,80){\vector(-1,0){1}}

\dashline{3}(356,80)(379,80)
\put(356,80){\vector(-1,0){1}}

\dashline{3}(76,130)(99,130)
\put(76,130){\vector(-1,0){1}}

\dashline{3}(36,130)(59,130)
\put(36,130){\vector(-1,0){1}}

\dashline{3}(116,130)(139,130)
\put(116,130){\vector(-1,0){1}}

\dashline{3}(156,130)(179,130)
\put(156,130){\vector(-1,0){1}}

\dashline{3}(196,130)(219,130)
\put(196,130){\vector(-1,0){1}}

\dashline{3}(236,130)(259,130)
\put(236,130){\vector(-1,0){1}}

\dashline{3}(276,130)(299,130)
\put(276,130){\vector(-1,0){1}}

\dashline{3}(316,130)(339,130)
\put(316,130){\vector(-1,0){1}}

\dashline{3}(356,130)(379,130)
\put(356,130){\vector(-1,0){1}}

\dashline{3}(76,40)(99,40)
\put(76,40){\vector(-1,0){1}}

\dashline{3}(36,40)(59,40)
\put(36,40){\vector(-1,0){1}}

\dashline{3}(116,40)(139,40)
\put(116,40){\vector(-1,0){1}}

\dashline{3}(156,40)(179,40)
\put(156,40){\vector(-1,0){1}}

\dashline{3}(196,40)(219,40)
\put(196,40){\vector(-1,0){1}}

\dashline{3}(236,40)(259,40)
\put(236,40){\vector(-1,0){1}}

\dashline{3}(276,40)(299,40)
\put(276,40){\vector(-1,0){1}}

\dashline{3}(316,40)(339,40)
\put(316,40){\vector(-1,0){1}}

\dashline{3}(356,40)(379,40)
\put(356,40){\vector(-1,0){1}}

\dashline{3}(5,40)(19,40)
\put(5,40){\vector(-1,0){1}}

\dashline{3}(5,80)(19,80)
\put(5,80){\vector(-1,0){1}}

\dashline{3}(5,130)(19,130)
\put(5,130){\vector(-1,0){1}}

\dashline{3}(395,40)(409,40)
\put(395,40){\vector(-1,0){1}}

\dashline{3}(395,80)(409,80)
\put(395,80){\vector(-1,0){1}}

\dashline{3}(395,130)(409,130)
\put(395,130){\vector(-1,0){1}}

\dashline{3}(376,150)(399,150)
\put(376,150){\vector(-1,0){1}}

\dashline{3}(336,150)(359,150)
\put(336,150){\vector(-1,0){1}}

\dashline{3}(296,150)(319,150)
\put(296,150){\vector(-1,0){1}}

\dashline{3}(256,150)(279,150)
\put(256,150){\vector(-1,0){1}}

\dashline{3}(216,150)(239,150)
\put(216,150){\vector(-1,0){1}}

\dashline{3}(176,150)(199,150)
\put(176,150){\vector(-1,0){1}}

\dashline{3}(136,150)(159,150)
\put(136,150){\vector(-1,0){1}}

\dashline{3}(96,150)(119,150)
\put(96,150){\vector(-1,0){1}}

\dashline{3}(56,150)(79,150)
\put(56,150){\vector(-1,0){1}}

\dashline{3}(16,150)(39,150)
\put(16,150){\vector(-1,0){1}}

\dashline{3}(376,100)(399,100)
\put(376,100){\vector(-1,0){1}}

\dashline{3}(336,100)(359,100)
\put(336,100){\vector(-1,0){1}}

\dashline{3}(296,100)(319,100)
\put(296,100){\vector(-1,0){1}}

\dashline{3}(256,100)(279,100)
\put(256,100){\vector(-1,0){1}}

\dashline{3}(216,100)(239,100)
\put(216,100){\vector(-1,0){1}}

\dashline{3}(176,100)(199,100)
\put(176,100){\vector(-1,0){1}}

\dashline{3}(136,100)(159,100)
\put(136,100){\vector(-1,0){1}}

\dashline{3}(96,100)(119,100)
\put(96,100){\vector(-1,0){1}}

\dashline{3}(56,100)(79,100)
\put(56,100){\vector(-1,0){1}}

\dashline{3}(16,100)(39,100)
\put(16,100){\vector(-1,0){1}}

\dashline{3}(376,60)(399,60)
\put(376,60){\vector(-1,0){1}}

\dashline{3}(336,60)(359,60)
\put(336,60){\vector(-1,0){1}}

\dashline{3}(296,60)(319,60)
\put(296,60){\vector(-1,0){1}}

\dashline{3}(256,60)(279,60)
\put(256,60){\vector(-1,0){1}}

\dashline{3}(216,60)(239,60)
\put(216,60){\vector(-1,0){1}}

\dashline{3}(176,60)(199,60)
\put(176,60){\vector(-1,0){1}}

\dashline{3}(136,60)(159,60)
\put(136,60){\vector(-1,0){1}}

\dashline{3}(96,60)(119,60)
\put(96,60){\vector(-1,0){1}}

\dashline{3}(56,60)(79,60)
\put(56,60){\vector(-1,0){1}}

\dashline{3}(16,60)(39,60)
\put(16,60){\vector(-1,0){1}}

\dashline{3}(376,20)(399,20)
\put(376,20){\vector(-1,0){1}}

\dashline{3}(336,20)(359,20)
\put(336,20){\vector(-1,0){1}}

\dashline{3}(296,20)(319,20)
\put(296,20){\vector(-1,0){1}}

\dashline{3}(256,20)(279,20)
\put(256,20){\vector(-1,0){1}}

\dashline{3}(216,20)(239,20)
\put(216,20){\vector(-1,0){1}}

\dashline{3}(176,20)(199,20)
\put(176,20){\vector(-1,0){1}}

\dashline{3}(136,20)(159,20)
\put(136,20){\vector(-1,0){1}}

\dashline{3}(96,20)(119,20)
\put(96,20){\vector(-1,0){1}}

\dashline{3}(56,20)(79,20)
\put(56,20){\vector(-1,0){1}}

\dashline{3}(16,20)(39,20)
\put(16,20){\vector(-1,0){1}}

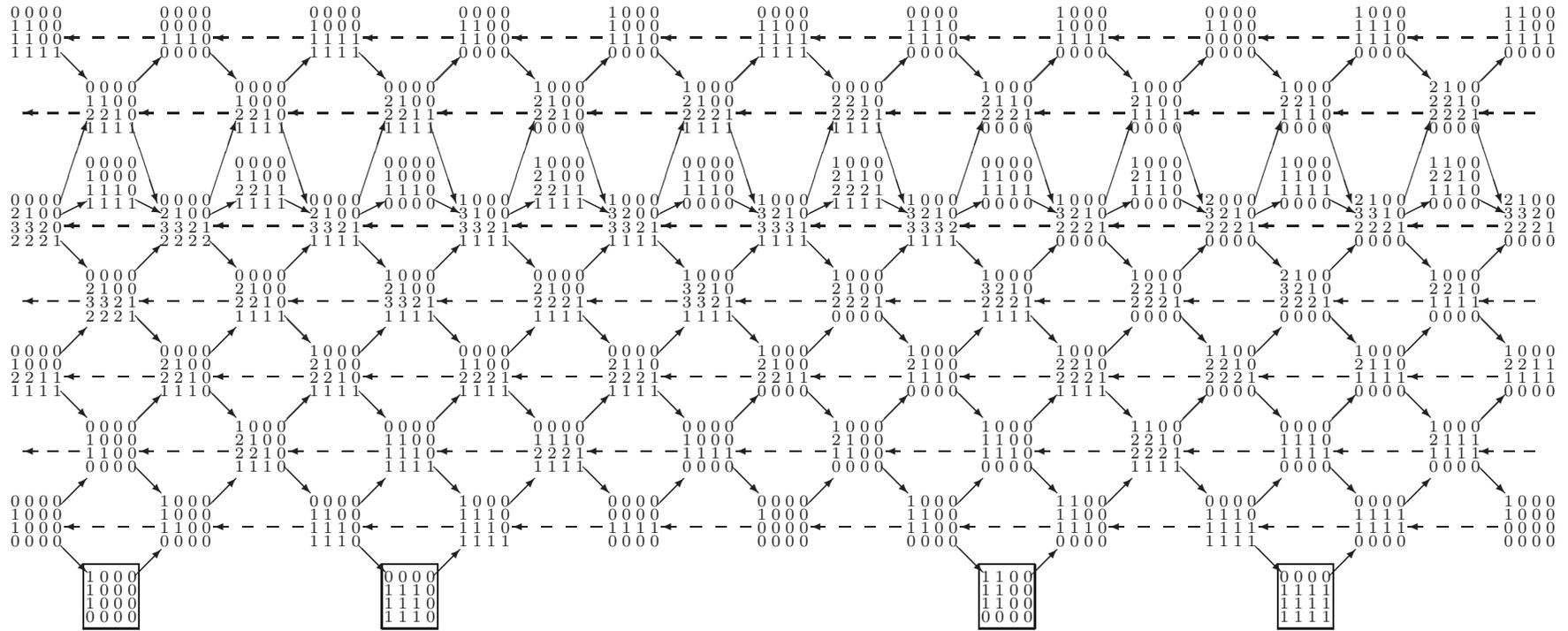
\put(120,-30){Figure 6. Auslander-Reiten quiver of $\Gproj(\Lambda_3)$ with $\Lambda$ hereditary of type $\A_4$}
\end{picture}

\end{center}

\end{landscape}

\subsection{} Finally, we give an example with $\Lambda$ not hereditary.
 Let $Q$ be the quiver as Figure 3 shows. Let $A=KQ/I$ be the algebra with $I$ the ideal of $KQ$ generated by the relation $\varepsilon_1^2,\varepsilon_2^2,\varepsilon_3^2$, $\beta\alpha$, $\beta\varepsilon_2-\varepsilon_3\beta$ and $\alpha\varepsilon_1- \varepsilon_2\alpha$.
 Let $\Lambda=K(1\stackrel{\alpha}{\longleftarrow}2\stackrel{\beta}{\longleftarrow}3)/\langle \beta\alpha\rangle$. Then $A\cong \Lambda_2=\Lambda\otimes_K K[X]/(X^2)$.  Note that $\Lambda$ is a tilted algebra of type $\A_3$. The Auslander-Reiten quiver of $\Lambda_2$ is displayed in Figure 7.
The modules on the leftmost column of the figure has to be identified with the corresponding module on the rightmost one.

\begin{center}\setlength{\unitlength}{0.6mm}
 \begin{picture}(120,100)(0,0)
\put(0,20){\tiny $\setlength{\arraycolsep}{1pt}\renewcommand{\arraystretch}{0.5}{\begin{array}{ccc} 0&0&0\\ 1&0&0\\ 1&1&0\\ 0&1&1  \end{array} }$ }

\put(0,70){\tiny $\setlength{\arraycolsep}{1pt}\renewcommand{\arraystretch}{0.5}{\begin{array}{ccc} 0&0&0\\ 0&0&0\\ 1&1&0\\ 0&0&0  \end{array} }$ }

\put(20,0){\tiny $\setlength{\arraycolsep}{1pt}\renewcommand{\arraystretch}{0.5}{\begin{array}{ccc} 0&0&0\\ 0&0&0\\ 0&1&1\\ 0&1&1  \end{array} }$ }

\put(20,40){\tiny $\setlength{\arraycolsep}{1pt}\renewcommand{\arraystretch}{0.5}{\begin{array}{ccc} 0&0&0\\ 1&0&0\\ 1&1&0\\ 0&0&0  \end{array} }$ }

\put(40,20){\tiny $\setlength{\arraycolsep}{1pt}\renewcommand{\arraystretch}{0.5}{\begin{array}{ccc} 0&0&0\\ 0&0&0\\ 0&1&1\\ 0&0&0  \end{array} }$ }
\put(40,50){\tiny $\setlength{\arraycolsep}{1pt}\renewcommand{\arraystretch}{0.5}{\begin{array}{ccc} 0&0&0\\ 1&1&0\\ 1&1&0\\ 0&0&0  \end{array} }$ }
\put(40,70){\tiny $\setlength{\arraycolsep}{1pt}\renewcommand{\arraystretch}{0.5}{\begin{array}{ccc} 0&0&0\\ 1&0&0\\ 0&0&0\\ 0&0&0  \end{array} }$ }
\put(60,40){\tiny $\setlength{\arraycolsep}{1pt}\renewcommand{\arraystretch}{0.5}{\begin{array}{ccc} 0&0&0\\ 1&1&0\\ 0&1&1\\ 0&0&0  \end{array} }$ }
\put(60,90){\tiny $\setlength{\arraycolsep}{1pt}\renewcommand{\arraystretch}{0.5}{\begin{array}{ccc} 1&0&0\\ 1&0&0\\ 0&0&0\\ 0&0&0  \end{array} }$ }
\put(80,70){\tiny $\setlength{\arraycolsep}{1pt}\renewcommand{\arraystretch}{0.5}{\begin{array}{ccc} 1&0&0\\ 1&1&0\\ 0&1&1 \\ 0&0&0 \end{array} }$ }

\put(80,20){\tiny $\setlength{\arraycolsep}{1pt}\renewcommand{\arraystretch}{0.5}{\begin{array}{ccc}  0&0&0\\ 1&1&0\\ 0&0&0 \\0&0&0 \end{array} }$ }

\put(51,26){\vector(1,1){9}}

\put(51,47){\vector(2,-1){9}}

\put(51,65){\vector(1,-2){9}}

\put(11,65){\vector(1,-2){9}}

\put(11,26){\vector(1,1){9}}

\put(51,76){\vector(1,1){9}}
\put(31,6){\vector(1,1){9}}

\put(71,36){\vector(1,-1){9}}
\put(31,36){\vector(1,-1){9}}
\put(11,16){\vector(1,-1){9}}
\put(71,86){\vector(1,-1){9}}

\put(31,43){\vector(2,1){9}}

\put(31,47){\vector(1,2){9}}
\put(71,47){\vector(1,2){9}}

\dashline{3}(40,20)(12,20)
\put(12,20){\vector(-1,0){1}}

\dashline{3}(80,20)(52,20)
\put(52,20){\vector(-1,0){1}}

\dashline{3}(40,70)(12,70)
\put(12,70){\vector(-1,0){1}}

\dashline{3}(80,70)(52,70)
\put(52,70){\vector(-1,0){1}}

\dashline{3}(5,40)(20,40)
\put(5,40){\vector(-1,0){1}}

\dashline{3}(72,40)(85,40)
\put(72,40){\vector(-1,0){1}}

\put(19.5,-7){\begin{picture}(10,20)
\put(0,0){\line(1,0){12}}
\put(0,0){\line(0,1){17}}
\put(12,0){\line(0,1){17}}
\put(0,17){\line(1,0){12}}
\end{picture}}

\put(39.5,43){\begin{picture}(10,20)
\put(0,0){\line(1,0){12}}
\put(0,0){\line(0,1){17}}
\put(12,0){\line(0,1){17}}
\put(0,17){\line(1,0){12}}
\end{picture}}

\put(59.5,83){\begin{picture}(10,20)
\put(0,0){\line(1,0){12}}
\put(0,0){\line(0,1){17}}
\put(12,0){\line(0,1){17}}
\put(0,17){\line(1,0){12}}
\end{picture}}

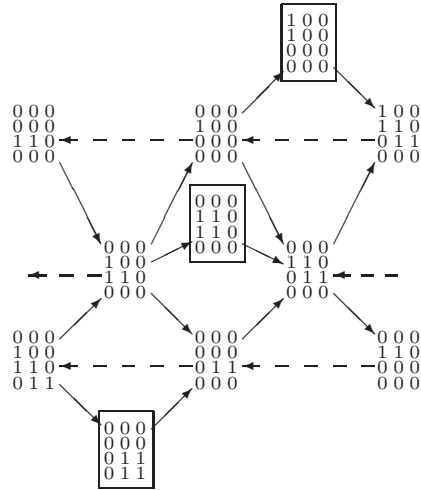
\put(-60,-20){Figure 7. Auslander-Reiten quiver for $\Gproj(\Lambda_2)$ with $\Lambda$ a tilted algebra of type $\A_3$.}

\end{picture}
\vspace{1.4cm}
\end{center}

As a comparison, let $B=KQ/I$ be the algebra with $I$ the ideal of $KQ$ generated by the relation $\varepsilon_1^2,\varepsilon_2^2,\varepsilon_3^3$, $\beta\varepsilon_2-\varepsilon_3\beta$ and $\alpha\varepsilon_1-\varepsilon_2\alpha$. Then $B\cong \Lambda_2=\Lambda\otimes_K K[X]/(X^2)$ for some hereditary algebra $\Lambda$ of type $\A_3$. The Auslander-Reiten quiver of $\Lambda_2$ is displayed in Figure 8.
The modules on the leftmost column of the figure has to be identified with the corresponding module on the rightmost one.

\begin{center}\setlength{\unitlength}{0.6mm}
 \begin{picture}(120,90)(0,0)
\put(0,20){\tiny $\setlength{\arraycolsep}{1pt}\renewcommand{\arraystretch}{0.5}{\begin{array}{ccc} 0&0&0\\ 1&0&0\\ 1&1&0  \end{array} }$ }

\put(0,60){\tiny $\setlength{\arraycolsep}{1pt}\renewcommand{\arraystretch}{0.5}{\begin{array}{ccc} 0&0&0\\ 0&0&0\\ 1&1&1  \end{array} }$ }

\put(20,0){\tiny $\setlength{\arraycolsep}{1pt}\renewcommand{\arraystretch}{0.5}{\begin{array}{ccc} 0&0&0\\ 1&1&0\\ 1&1&0  \end{array} }$ }

\put(20,40){\tiny $\setlength{\arraycolsep}{1pt}\renewcommand{\arraystretch}{0.5}{\begin{array}{ccc} 0&0&0\\ 1&0&0\\ 1&1&1  \end{array} }$ }

\put(40,20){\tiny $\setlength{\arraycolsep}{1pt}\renewcommand{\arraystretch}{0.5}{\begin{array}{ccc} 0&0&0\\ 1&1&0\\ 1&1&1  \end{array} }$ }

\put(40,60){\tiny $\setlength{\arraycolsep}{1pt}\renewcommand{\arraystretch}{0.5}{\begin{array}{ccc} 0&0&0\\ 1&0&0\\ 0&0&0  \end{array} }$ }

\put(60,0){\tiny $\setlength{\arraycolsep}{1pt}\renewcommand{\arraystretch}{0.5}{\begin{array}{ccc} 0&0&0\\ 1&1&1\\ 1&1&1  \end{array} }$ }

\put(60,40){\tiny $\setlength{\arraycolsep}{1pt}\renewcommand{\arraystretch}{0.5}{\begin{array}{ccc} 0&0&0\\ 1&1&0\\ 0&0&0  \end{array} }$ }

\put(60,80){\tiny $\setlength{\arraycolsep}{1pt}\renewcommand{\arraystretch}{0.5}{\begin{array}{ccc} 1&0&0\\ 1&0&0\\ 0&0&0  \end{array} }$ }

\put(80,60){\tiny $\setlength{\arraycolsep}{1pt}\renewcommand{\arraystretch}{0.5}{\begin{array}{ccc}  1&0&0\\ 1&1&0  \\0&0&0\end{array} }$ }

\put(80,20){\tiny $\setlength{\arraycolsep}{1pt}\renewcommand{\arraystretch}{0.5}{\begin{array}{ccc}  0&0&0\\ 1&1&1 \\0&0&0 \end{array} }$ }

\put(11,56){\vector(1,-1){9}}

\put(11,26){\vector(1,1){9}}
\put(11,16){\vector(1,-1){9}}

\put(51,56){\vector(1,-1){9}}

\put(51,26){\vector(1,1){9}}
\put(51,16){\vector(1,-1){9}}

\put(31,36){\vector(1,-1){9}}
\put(31,6){\vector(1,1){9}}
\put(31,46){\vector(1,1){9}}
\put(51,66){\vector(1,1){9}}
\put(71,36){\vector(1,-1){9}}
\put(71,6){\vector(1,1){9}}
\put(71,46){\vector(1,1){9}}
\put(71,76){\vector(1,-1){9}}

\put(59.5,75){\begin{picture}(10,20)
\put(0,0){\line(1,0){12}}
\put(0,0){\line(0,1){13}}
\put(12,0){\line(0,1){13}}
\put(0,13){\line(1,0){12}}
\end{picture}}
\put(59.5,-5){\begin{picture}(10,20)
\put(0,0){\line(1,0){12}}
\put(0,0){\line(0,1){13}}
\put(12,0){\line(0,1){13}}
\put(0,13){\line(1,0){12}}
\end{picture}}
\put(19.5,-5){\begin{picture}(10,20)
\put(0,0){\line(1,0){12}}
\put(0,0){\line(0,1){13}}
\put(12,0){\line(0,1){13}}
\put(0,13){\line(1,0){12}}
\end{picture}}

\dashline{3}(40,20)(12,20)
\put(12,20){\vector(-1,0){1}}

\dashline{3}(80,20)(52,20)
\put(52,20){\vector(-1,0){1}}

\dashline{3}(40,60)(12,60)
\put(12,60){\vector(-1,0){1}}

\dashline{3}(80,60)(52,60)
\put(52,60){\vector(-1,0){1}}

\dashline{3}(60,40)(32,40)
\put(32,40){\vector(-1,0){1}}

\dashline{3}(5,40)(20,40)
\put(5,40){\vector(-1,0){1}}

\dashline{3}(72,40)(87,40)
\put(72,40){\vector(-1,0){1}}

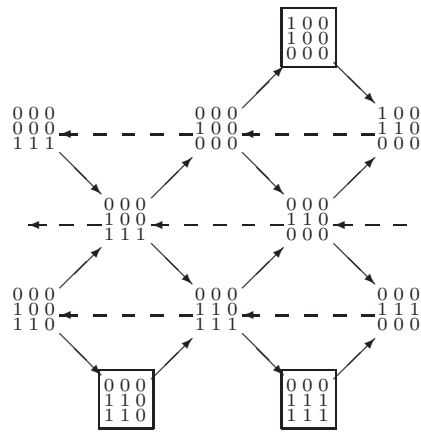
\put(-40,-20){Figure 8. Auslander-Reiten quiver for $\Gproj(\Lambda_2)$ with $\Lambda$ hereditary of type $\A_3$}

\end{picture}
\vspace{1.8cm}
\end{center}

\end{document}